\documentclass[10pt]{tran-l}
\usepackage[a4paper,margin=2.0cm]{geometry}

\usepackage{amssymb}
\usepackage{setspace}
\usepackage[shortlabels]{enumitem}
\usepackage{xcolor}
\usepackage{amsmath}
\usepackage{amsthm}
\usepackage{hyperref}
\usepackage{arydshln}
\usepackage{tikz}

\newtheorem{theorem}{Theorem}[section]
\newtheorem{lemma}[theorem]{Lemma}
\newtheorem{corollary}[theorem]{Corollary}
\newtheorem{proposition}[theorem]{Proposition}

\theoremstyle{definition}

\newtheorem{notation}[theorem]{Notation}

\theoremstyle{remark}
\newtheorem{remark}[theorem]{Remark}

\numberwithin{equation}{section}

\renewcommand{\L}{\mathbb L}
\renewcommand{\P}{\mathbb P}
\renewcommand{\S}{\mathbb S}
\newcommand{\D}{\mathbb D}
\newcommand{\E}{\mathbb E}
\newcommand{\M}{\mathbb M}
\newcommand{\f}    [1]{\mathbb{F}_{#1}}

\newcommand{\Gcd}  [2]{{\mathsf{gcd}({#1},{#2})}}

\newcommand{\squares}[1]{{(#1^\times)^2}}

\newcommand{\DE}[2]{{\mathsf{D}_{#1}^{{#2}}}}
\newcommand{\ProjectiveLine}[0]{{\mathcal{P}^1}}
\newcommand{\Family}[1]{\mathcal{S}}

\newcommand{\qbar}{{\overline{q}}}

\renewcommand{\aa}{\alpha}
\newcommand{\bb}{\beta}
\newcommand{\cc}{\gamma}
\newcommand{\dd}{\delta}

\newcommand{\N}{\mathbb N}

\newcommand{\F}{\mathbb F}

\newcommand{\ML}{{\mathcal L}}
\newcommand{\MN}{{\mathcal N}}

\newcommand{\faruk}{\textcolor[rgb]{0.1,0.1,0.9}}

\DeclareMathOperator{\Gal}{Gal}
\DeclareMathOperator{\End}{End}

\DeclareMathOperator{\GL}{GL}

\DeclareMathOperator{\Aut}{Aut}

\DeclareMathOperator{\diag}{diag}

\newcommand{\parskipsize}{.4em}
\begin{document}

\title{An exponential bound on the number of non-isotopic commutative semifields}
\author{Faruk G\"{o}lo\u{g}lu}
\address{Charles University in Prague}
\curraddr{}
\email{Faruk.Gologlu@mff.cuni.cz}
\thanks{}

\author{Lukas K\"olsch}
\address{University of Rostock, Germany; University of South Florida }
\curraddr{}
\email{lukas.koelsch.math@gmail.com}
\thanks{}


\date{}

\dedicatory{}

\begin{abstract}
We show that the number of non-isotopic commutative semifields of odd order
$p^{n}$ is exponential in $n$ when $n = 4t$ and $t$ is not a power of $2$. 
We introduce a new family of commutative semifields and a method for proving
isotopy results on commutative semifields that we use to deduce the 
aforementioned bound. The previous best bound on the number of non-isotopic
commutative semifields of odd order was quadratic in $n$ and was given by 
Zhou and Pott [Adv. Math. \textbf{234} (2013)]. Similar bounds in the case 
of even order were given in Kantor [J. Algebra \textbf{270} (2003)] and  
Kantor and Williams [Trans. Amer. Math. Soc. \textbf{356} (2004)].

\vspace{0.2cm}

\faruk{\textbf{Updated version: Typos in 
\hyperref[table_biproj]{Table 1} and 
\hyperref[table_comm]{Table 2} corrected.
\hyperref[secaut]{Author Contributions} (CRedit) 
added.}}

\end{abstract}

\maketitle


\section{Introduction}

In this paper, we show that the number $N_{p^{n}}$ of non-isotopic 
commutative semifields of odd order $p^{n}$ is exponential in 
$n$ when $n = 4t$. To be precise, we prove for every odd prime $p$,
\[
	N_{p^n} \ge \frac{(\sigma(n)-1)(p^{n/4}-1)}{2n},
\]
when $\nu_2(n) \ge 2$, where we denote by $\sigma(n)$ the odd part of 
an integer $n$ (i.e., $\sigma(n) = n/2^{\nu_2(n)}$), and by $\nu_2(n)$ 
the $2$-adic valuation of $n$ (i.e., $2^{\nu_2(n)} | n$ and $2^{\nu_2(n)+1} \nmid n$).
For odd $p$, the previous best bound on $N_{p^n}$ was quadratic in $n$ and was 
proved in \cite[Corollary 1]{ZP13}:
\[
	N_{p^n} \ge \frac{n(\sigma(n)-1)}{8} + cn,
\]
when $\nu_2(n) \ge 1$ and $c$ a constant. 
When $p$ and $n$ are odd, the known number for $N_{p^n}$ is linear in $n$.
The problem of determining
whether the number $N_{p^n}$ can be bounded by a polynomial in $n$ has been 
described  \cite[p. 180]{Pott16} as 
``the main problem in connection with commutative semifields of [odd] order $p^n$.''  
Note that it is impossible to find families with exponentially many non-isotopic
commutative semifields of order $p^n$ for arbitrary $p,n$. Indeed, by a result 
of Menichetti~\cite[Corollary 33]{Menichetti96}, all commutative semifields of order $p^n$ with 
$n$ prime and $p$ large enough are isotopic to the finite field or a twisted 
field (see Section~\ref{sec:albert}). It is thus impossible to give an exponential 
count for all $p,n$.
The problem in the characteristic $2$ case was solved almost
two decades ago. Kantor \cite[Theorem 1.1]{Kantor03} showed that the number of 
non-isotopic commutative semifields of order $2^{km}$ is at least
\[
	N_{2^{km}} \ge \frac{2^{km(\rho(m)-1)}}{k^2m^4},
\]
when $m > 1$ is odd and $m$ is not a power of $3$
(where we denote by $\rho(m)$ the number of prime factors of $m$ counting multiplicities), 
using a construction by
Kantor and Williams \cite[Theorem 1.7]{KW}. In these papers, finding a 
large number of (commutative) semifields in odd
characteristic and finding a general approach to proving non-isotopy were posed 
as important open problems \cite[p. 936]{KW},\cite[p. 112]{Kantor03}.

To prove the bound we introduce a new family of commutative semifields. These semifields satisfy a property that 
we call {\em biprojectivity}, which also applies to many known semifields of square 
order. The biprojective structure allows us to develop a technique of determining
isotopy between semifields. This technique is key to proving the exponential bound on 
non-isotopic commutative semifields.

In Section \ref{sec_prem}, we give the preliminaries. In Section \ref{sec_biproj}, we 
define biprojective semifields and give a quick survey on known commutative semifields
and their counts. Section \ref{sec_families} is devoted to proving the semifield property
of our family (Theorem~\ref{thm_family1}). 
Section \ref{sec_method} introduces our technique of proving isotopy between semifields
(Theorem~\ref{thm:equivalence}).
In Section \ref{sec_counts1}, we give the number of non-isotopic
semifields arising from our family (Theorem~\ref{thm:f1_equiv}).
Section \ref{sec_counts1} contains our main result that the number of non-isotopic 
commutative semifields of odd order $p^{n}$ is exponential in 
$n$ (Corollaries~\ref{cor:f1_equiv2} and \ref{cor:exponential}). 
In Section \ref{sec_nuclei}, we compute the nuclei associated to our semifields
(Theorem~\ref{thm_nucleus}). Finally, in Section \ref{sec_equiv},
we show that our semifields are indeed new and not isotopic to most known semifields 
(Theorem \ref{thm:inequiv_distinct}).


\section{Preliminaries} \label{sec_prem}

A \textbf{finite semifield} $\S = (S,+,\circ)$ is a set $S$ equipped with two operations $(+,\circ)$
satisfying the following axioms. 
\begin{enumerate}
\item[(S1)] $(S,+)$ is a group.
\item[(S2)] For all $x,y,z \in S$,
\begin{itemize}
\item $x\circ (y+z) = x \circ y + x \circ z$,
\item $(x+y)\circ z = x \circ z + y \circ z$.
\end{itemize}
\item[(S3)] For all $x,y \in S$, $x \circ y = 0$ implies $x=0$ or $y=0$.
\item[(S4)] There exists $\epsilon \in S$ such that $x\circ \epsilon = x = \epsilon \circ x$.
\end{enumerate}

In this paper, we will be interested only in finite semifields. 
Henceforth, when we say a semifield we will mean a finite semifield.
An algebraic object satisfying the first three of the above axioms is called 
a \textbf{pre-semifield}. If $\P = (P,+,\circ)$ is a pre-semifield, then $(P,+)$ is 
an elementary abelian $p$-group \cite[p. 185]{Knuth65}, and $(P,+)$ can be viewed as
an $n$-dimensional $\f{p}$-vector space $\F_p^n$. 
If $\circ$ is associative then $\S$ is the finite field $\f{p^n}$ by 
Wedderburn's theorem which states that a finite division ring is a field. 
By a result of Menichetti (known as Kaplansky's conjecture \cite{Menichetti})
when $n > 2$, there exist {\em proper} semifields of odd order $p^n$ where 
$\circ$ is non-associative. 
There are no proper semifields of order $2^3$. For $n > 3$, there 
exists proper semifields of order $2^n$ \cite{Knuth65}. 
A pre-semifield $\P = (\F_p^n,+,\circ)$ 
can be converted to a semifield $\S = (\F_p^n,+,\ast)$ using {\em Kaplansky's trick} 
by defining the new multiplication as
\[
	(x \circ e) \ast (e \circ y) = (x \circ y),
\]
for any nonzero element $e \in \F_p^n$, making $(e \circ e)$ the multiplicative 
identity of $\S$. A pre-semifield is an $\f{p}$-algebra, thus the multiplication
is bilinear. Therefore we have $\f{p}$-bilinear 
$B : \F_p^n \times \F_p^n \to \F_p^n$, satisfying
\[
	B(x,y) = x \circ y,
\]
and $\f{p}$-linear left and right multiplications $L_x,R_y : \F_p^n \to \F_p^n$, with
\[
L_x(y) := B(x,y) =: R_y(x).
\]
The mapping $L_x$ (resp. $R_y$) is a bijection whenever $x \ne 0$ (resp. $y \ne 0$) 
by (S3). Thus,
\[
R_e(x) \ast L_e(y) = x \circ y.
\]
Two pre-semifields $\P_1 = (\F_p^n,+,\circ_1)$ and $\P_2 = (\F_p^n,+,\circ_2)$ 
are said to be \textbf{isotopic} if there exist $\f{p}$-linear bijections 
$L,M$ and $N$ of $\F_p^n$ satisfying
\[
N(x \circ_1 y) = L(x) \circ_2 M(y).
\]
Such a triple $\gamma = (N,L,M)$ is called an \textbf{isotopism} between $\P_1$ and $\P_2$. 
If additionally $L=M$ holds, we call $\gamma$ a \textbf{strong isotopism} and $\P_1$ and $\P_2$ \textbf{strongly isotopic}. 
Isotopisms between a pre-semifield $\P$ and itself are called \textbf{autotopisms}.
Thus the pre-semifield $\P$ and the corresponding semifield $\S$ constructed 
by Kaplansky's trick are isotopic and even strongly isotopic if $\P$ is commutative.
Isotopy of pre-semifields is an equivalence relation and the isotopism class
of a pre-semifield $\P$ is denoted by $[\P]$. 
Semifields coordinatize projective planes and different semifields coordinatize
isomorphic planes if and only if they are isotopic 
(\cite{Albert60}, see \cite[Section 3]{Knuth65} for a detailed
treatment). Semifields are further equivalent to maximum rank distance codes with 
certain parameters (see e.g.~\cite{sheekey}) and can be used to construct 
relative difference sets~(see \cite{pott2014semifields}).
Associative substructures of a semifield $\S = (\F_p^n,+,\ast)$, 
namely the \textbf{left, middle and right nuclei}, are defined as follows:
\begin{align*}
\N_l(\S) &:= \{ x \in \F_p^n \ : \ (x \ast y) \ast z = x \ast (y \ast z), \ y,z \in \F_p^n \},\\
\N_m(\S) &:= \{ y \in \F_p^n \ : \ (x \ast y) \ast z = x \ast (y \ast z), \ x,z \in \F_p^n \},\\
\N_r(\S) &:= \{ z \in \F_p^n \ : \ (x \ast y) \ast z = x \ast (y \ast z), \ x,y \in \F_p^n \}.
\end{align*}
It is easy to check that $\N_l(\S),\N_m(\S),\N_r(\S) \subseteq \f{p^n}$ are 
finite fields and if $\S$ is commutative then $\N_l(\S) = \N_r(\S)$. 
Nuclei are isotopy invariants for semifields. Since every pre-semifield $\P \in [\S]$
for some semifield $\S$, the nuclei can be thought to extend to pre-semifields. 
Thus, when we speak of the nuclei of a pre-semifield $\P$ we mean the nuclei of 
the isotopic semifield $\S$.

Let $\End(\f{p}^n)$ denote the $\f{p}$-linear endomorphisms of the vector space 
$\f{p}^n$. Every $\f{p}$-linear mapping $L \in \End(\F_p^n)$ can be written 
uniquely as 
a \textbf{linearized polynomial} 
\[
L(x) = \sum_{i = 0}^{n-1} b_i x^{p^i},  
\]
in the polynomial ring $\f{p^n}[x]$. We will not make distinction between mappings 
and the polynomials. Let $p$ be an odd prime and 
consider the polynomials from $\F_{p^n}[x]$ of the form
\[
F(x) = \sum_{0 \le i,j < n} a_{ij} x^{p^i+p^j}.
\]
These polynomials are called \textbf{Dembowski-Ostrom (DO)} polynomials. The 
\textbf{polarization} of a DO polynomial $F$ is defined as
\[
\Delta_F(x,y) = F(x+y) -F(x) -F(y). 
\]
The mapping $\Delta_F : \F_p^n \times \F_p^n \to \F_p^n$ is symmetric
and $\f{p}$-bilinear, thus if $\Delta_F(x,a) = 0$ implies $x = 0$ for all 
$a \in \F_{p^n}^\times = \F_{p^n} \setminus \{0\}$, then $\Delta_F(x,y)$ describes a 
commutative pre-semifield multiplication \cite{DO}. Conversely, by a 
counting argument, every commutative pre-semifield multiplication can be 
written as $\Delta_F(x,y)$ for some DO polynomial $F$ \cite{CM}. In that 
case we will call $F$ a \textbf{planar DO polynomial/mapping}. 
Strong isotopy between pre-semifields
can be recognized also in the corresponding planar DO polynomials:
\begin{theorem}{\cite[Theorem 3.5.]{coulterhenderson}} \label{thm:eaequiv}
Let $F,G \in \f{p^n}[x]$ be planar DO polynomials and 
$\P_1$, $\P_2$ be the corresponding pre-semifields. Then 
$\P_1$ and $\P_2$ are strongly isotopic via an isotopism 
$\gamma = (N,L,L)$ if and only if $F=NGL^{-1}$. 
\end{theorem}
Consequently, we say that two planar DO polynomials $F,G$ are \textbf{equivalent} 
if bijective linear mappings $L_1,L_2$ exist such that $F=L_1 G L_2$. Note that 
this type of equivalence is the most general equivalence known to preserve the 
planarity of a DO polynomial, see~\cite{KP08}.  


\section{Biprojective planar mappings and commutative semifields} \label{sec_biproj}

In this paper we are interested in planar DO polynomials of a specific form.
Let $\F = \f{p^n}$ be a finite field of square odd order and $\M=\f{p^m}$ with 
$n = 2m$. Let 
\begin{equation*}
F(x,y) = \left(f(x,y),g(x,y)\right),
\end{equation*}
where
\begin{align*}
f(x,y) &= a_0 x^{q+1} + b_0 x^q y + c_0 x y^q + d_0 y^{q+1},\\
g(x,y) &= a_1 x^{r+1} + b_1 x^r y + c_1 x y^r + d_1 y^{r+1},
\end{align*}
with $q = p^k, r = p^l, 1 \le k,l \le m$.
We will call
$f(x,y)$ a $q$\textbf{-biprojective polynomial} and 
$(x,y) \mapsto F(x,y)$ a $(q,r)$\textbf{-biprojective mapping} (of $\M \times \M$).
Note that $F(x,y)$ is a {\em $(q,r)$-biprojective polynomial pair}. We will
not make any distinction between the polynomials and the mappings defined by them.
We also let $\qbar = p^{m-k}$ and $\overline{r} = p^{m-l}$, so that
$q\qbar \equiv r \overline{r} \equiv 1 \pmod{p^m-1}$.
We are going to use the shorthand notation
\begin{align*}
f(x,y) 	&= (a_0,b_0,c_0,d_0)_q,\\
g(x,y)	&= (a_1,b_1,c_1,d_1)_r.
\end  {align*}
We are going to refer to $f$ and $g$ as (left and right) \textbf{components} of $F$. We refer the reader to \cite{Bluher} for projective polynomials over finite fields.

The polarization of a planar $(q,r)$-biprojective mapping defines a
\textbf{$(q,r)$-biprojective (commutative) pre-semifield} $\P = (\M\times\M,+,\ast)$. It is easy to see that
both components correspond to homogeneous operations due to biprojectivity:
\begin{align} 
(x,y) \ast (y,v) = ((&a_0 u + b_0 v) x^q + (a_0 u^q + c_0 v^q) x \label{eq_polarization}
	 			             + (c_0 u + d_0 v) y^q + (b_0 u^q + d_0 v^q) y,\\
										(&a_1 u + b_1 v) x^r + (a_1 u^r + c_1 v^r) x \nonumber
	 			             + (c_1 u + d_1 v) y^r + (b_1 u^r + d_1 v^r) y).
\end{align}
Define

\noindent\begin{minipage}{0.5\textwidth}
\begin{align*}
 \DE{f}{0}(x,y) &= b_0 x^q + c_0 x + d_0 y^q + d_0 y,\\  
 \DE{g}{0}(x,y) &= b_1 x^r + c_1 x + d_1 y^r + d_1 y,
\end{align*}
\end{minipage}
\noindent\begin{minipage}{0.5\textwidth}
\begin{align*}
 \DE{f}{\infty}(x,y) &= a_0 x^q + a_0 x + c_0 y^q + b_0 y,\\
 \DE{g}{\infty}(x,y) &= a_1 x^r + a_1 x + c_1 y^r + b_1 y,
\end{align*}
\end{minipage}
\vskip1em

and for $u \in \ProjectiveLine(\M) \setminus \{0,\infty\}$,
\begin{align*}
 \DE{f}{u}(x,y) &= (a_0 u + b_0) x^q + (a_0 u^q + c_0) x 
	 			+ (c_0 u + d_0) y^q + (b_0 u^q + d_0) y,\\
 \DE{g}{u}(x,y) &= (a_1 u + b_1) x^r + (a_1 u^r + c_1) x 
	 			+ (c_1 u + d_1) y^r + (b_1 u^r + d_1) y.
\end{align*} 

The following lemma is straightforward.

\begin{lemma}\label{lem_PN}
Let $(x,y) \mapsto F(x,y) = (f(x,y),g(x,y))$ be a $(q,r)$-biprojective mapping of $\M \times \M$. 
Then $F$ is planar if and only if the pair of equations
\[
\DE{f}{u}(x,y) = 0 = \DE{g}{u}(x,y)
\]
has exactly one solution for each $u \in \ProjectiveLine(\M)$.
\end{lemma}
\begin{proof}
We need to show that the polarization $\Delta_F((x,y),(u,v)) = (x,y) \ast (u,v) = 0$ has a unique
zero for each $(u,v) \in \M \times \M \setminus (0,0)$ if and only if
$\DE{f}{w}(x,y) = 0 = \DE{g}{w}(x,y)$ has a unique solution for each 
$w \in \ProjectiveLine(\M)$. Inspecting Eq. \eqref{eq_polarization}, one immediately sees that 
the case $v = 0$ corresponds to 
$\DE{f}{\infty}(x,y) = 0 = \DE{g}{\infty}(x,y)$
after applying
$x \mapsto xu$ and $y \mapsto yu$. For $v \in \M^\times$,
apply $x \mapsto xv$, $y \mapsto yv$ and $u \mapsto uv$ to get the remaining
cases 
$\DE{f}{w}(x,y) = 0 = \DE{g}{w}(x,y)$
for $w \in \M$.
\end{proof}

In the following we will show that many known semifields fall into the
$(q,r)$-biprojective setting.

\subsection{Dickson semifields $\mathcal{D}$}

Dickson introduced the commutative semifields $\S = (\M \times \M, +, \circ)$ 
with
\[
	(x,y) \circ (u,v) = (xu + ay^qv^q, xv + yu)
\]
where $q = p^k$ with $0 < k < l$ and $a \in \M^\times \setminus \squares{\M}$. 
Note that the isotopic multiplication 
\[
	(x,y) \ast (u,v) = (xu + ayv, xv^\qbar + y^\qbar u)
\]
is $(1,\qbar)$-biprojective and isotopic to the polarization of the 
$(1,q)$-biprojective planar mapping
\[
	F_\mathcal{D} = \left((1,0,0,a)_1,(0,1,0,0)_q\right).
\]
Different choices for $a \in \M^\times \setminus \squares{\M}$ produce isotopic 
semifields and there are a total of $\lfloor \frac{n}{4} \rfloor$ non-isotopic Dickson semifields 
\cite[p.107]{Kantor03}.

\subsection{Albert's generalized twisted fields $\mathcal{A}$} \label{sec:albert}

Albert introduced \cite{Albert61} a family of commutative and noncommutative semifields.
The commutative ones may be given as $\S = (\F,+,\circ)$ with
\[
	X \circ U = X^qU + U^qX,
\]
where $q = p^k$ with $0 < k < n$ satisfying $n/\Gcd k n$ odd. When $\F = \M(\xi)$
with $[\F : \M] = 2$, one can write $X = x\xi+y$ with $x,y \in \M$. One can choose
$\xi \in \F \setminus \M$ satisfying $\xi^2 = a \in \M^\times \setminus \squares{\M}$.
\begin{align*}
	(x\xi+y) \circ (u\xi+v) &= (x\xi+y)^q(u\xi+v) + (u\xi+v)^q(x\xi+y)\\
	                        &= \xi^{q+1} (x^qu + u^qx) + \xi^q (x^qv + u^qy)
													   + \xi       (y^qu + v^qx) +       (y^qv + v^qy)\\
													&= a^{(q+1)/2} (x^qu + u^qx) + a^{(q-1)/2}\xi(x^qv + u^qy)
													   + \xi       (y^qu + v^qx) +       (y^qv + v^qy).
\end{align*}
Identifying $\xi\M + \M$ with $\M \times \M$, and
\begin{align*}
  (x,y) \circ (u,v) &= \left(a^{(q-1)/2}(x^qv + u^qy) + (y^qu + v^qx), 
	                     \quad a^{(q+1)/2} (x^qu + u^qx) + (y^qv + v^qy)\right),
\end{align*}
is $(q,q)$-biprojective and isotopic to the polarization of the $(q,q)$-biprojective
planar mapping
\[
	F_\mathcal{A} = ((0,a^{(q-1)/2},1,0)_q,(a^{(q+1)/2},0,0,1)_q).
\]
Different choices for $a \in \M^\times \setminus \squares{\M}$ produce isotopic 
semifields and there are a total of $\frac{\sigma(n)-1}{2}$ non-isotopic generalized twisted 
fields \cite{Kantor03,Albert61}.

\subsection{Zhou-Pott semifields $\mathcal{ZP}$}

Zhou and Pott \cite{ZP13} gave a family of pre-semifields 
$S = (\M \times \M, +, \circ)$ given by
\[
(x,y) \circ (u,v) = (x^qu+u^qx + a(y^qv+yv^q)^r, xv+yu),
\]
where $a \in \M \setminus \squares{\M}$, $q = p^k$ and 
$r = p^j$ with $0 \le j,k \le m$ where $m/\Gcd{k}{m}$ is odd.
The isotopic multiplication
\[
(x,y) \ast (u,v) = (x^qu+u^qx + a(y^qv+yv^q), x^rv+yu^r),
\]
is $(q,r)$-biprojective and isotopic to the polarization of the $(q,r)$-biprojective
planar mapping
\[
 F_{\mathcal{ZP}} =  ((1,0,0,a)_q,(0,1,0,0)_r).
\]
Different choices for $a \in \M^\times \setminus \squares{\M}$ produce isotopic 
semifields and there are a total of $\left\lfloor\frac{\sigma(n)}{2}\right\rfloor \cdot
\left\lceil\frac{n}{4}\right\rceil$ non-isotopic $\mathcal{ZP}$ semifields \cite{ZP13}.

\subsection{Budaghyan-Helleseth semifields ($\mathcal{BH},\mathcal{ZW},\mathcal{LMPTB}$)}
These semifields were found in \cite{BH} and independently in \cite{ZW}. The commutative 
semifields given later in \cite{LMPT} and \cite{Bierbrauer16} were
shown to be isotopic to the previous ones \cite{MP12}. We note that Bierbrauer's construction 
in \cite{Bierbrauer16} gives also non-commutative semifields. We will use the definition from
\cite{Bierbrauer16}. Let $S = (\M \times \M, +, \circ)$ be the pre-semifield given by
\[
(x,y) \circ (u,v) = \left\{\begin{array}{ll}
(xv + yu, x^qu+xu^q +a          (y^qv+yv^q)) & \textrm{if } m/\Gcd{k}{m} \textrm{ is odd},\\
(xu +ayv, x^qv+yu^q +a^{(q-1)/2}(xv^q+y^qu)) & \textrm{if } m/\Gcd{k}{m} \textrm{ is even},
\end{array}
\right.
\]
where $a \in \M \setminus \squares{\M}$ and $q = p^k$ with $0 < k < m$. The 
pre-semifield multiplication is $(1,q)$-biprojective. Similarly, the corresponding 
$(1,q)$-biprojective planar mapping whose polarization is isotopic to $\S$ is given by
\begin{align*}
F_\mathcal{BH} = \left\{
\begin{array}{ll}
((0,0,1,0)_1,(1,0,0,a)_q)& \textrm{if } m/\Gcd{k}{m} \textrm{ is odd},\\ 
((1,0,0,a)_1,(0,1,a^{(q-1)/2},0)_q) & \textrm{if } m/\Gcd{k}{m} \textrm{ is even}.
\end{array}
\right.
\end{align*}

The number of non-isotopic semifields in this family is $\lfloor \frac{n}{4} \rfloor$ which is proved in \cite{FL18}.

\begin{table}[!ht] 
\noindent\begin{center} 
{\scriptsize
\begin{tabular}{|c|c|c|c|c|c|c|} 
\hline 
\textbf{Family} & \textbf{Planar Mapping} & $\#\S$ & \textbf{Notes} & $(\#\N_l,\#\N_m)$& \textbf{Count} & \textbf{Proved in}\\
\hline

$\F$ & $X^{2}$ & $p^{n}$ & & $(p^n,p^n)$ & $1$  &   \\ 
\cdashline{2-4}[.8pt/0.5pt]
& $((0,1,0,0)_1,(1,0,0,a)_1)$ & & $a \in \M \setminus \squares{\M}, n = 2m$& & & \\
\hline

$\mathcal{A}$   &        $X^{q+1}$                                &$p^{n}$  & 
\begin{tabular}{@{}c@{}} $q=p^k$, $0 < k < n$,\\ $\Gcd{k}{n}=d$, $n/d$ odd. \end{tabular} & 
$(p^d,p^d)$     & $\left\lfloor\frac{\sigma(n)-1}{2}\right\rfloor$  & \cite{Albert61}   \\ 
\cdashline{2-4}[.8pt/0.5pt]
                & $((0,a^{\frac{q-1}{2}},1,0)_q,(a^{\frac{q+1}{2}},0,0,1)_q)$ & & $a \in \M \setminus \squares{\M}, n = 2m$& & & \\
\hline 

$\mathcal{D}$   & $((1,0,0,a)_1,(0,1,0,0)_{q})$                   &$p^{2m}$ & 
\begin{tabular}{@{}c@{}}  $q=p^k$, $0 < k < m$,\\ $\Gcd{k}{m}=d$,\\ $a \in \M \setminus \squares{\M}$. \end{tabular} & 
$(p^d,p^m)$     &  $\left\lfloor\frac{n}{4}\right\rfloor$ & \cite{Dickson}  \\ 
\hline 

$\mathcal{ZP}$  & $((1,0,0,a)_q,(0,1,0,0)_r)$                     &$p^{2m}$ & 
\begin{tabular}{@{}c@{}} $q=p^k,r=p^j$, $0 < j,k < m$,\\ $\Gcd{k}{m}=d,\quad \Gcd{j,k}{m}=d'$,\\ $m/d$ odd,
 $a \in \M \setminus \squares{\M}$.\end{tabular} & 
$(p^{d'},p^d)$     & $\left\lfloor\frac{\sigma(n)-1}{2}\right\rfloor  \left\lfloor\frac{n}{4}\right\rfloor$        & \cite{ZP13}       \\ 
\hline 

$\mathcal{BH}$  & $((0,1,0,0)_1,(1,0,0,a)_q)$                     &$p^{2m}$ & 
\begin{tabular}{@{}c@{}} $q=p^k$, $0 < k < m$,\\ $\Gcd{k}{m}=d$, $m/d$ odd,\\ $a \in \M \setminus \squares{\M}$.\end{tabular} & 
$(p^d,p^{2d})$  & $\left\lfloor\frac{n}{4}\right\rfloor$ & \cite{BH,ZW,FL18}       \\ 
\cdashline{2-4}[.4pt/1pt]
                & $((1,0,0,a)_1,(0,1,a^{\frac{q-1}{2}},0)_q)$ & & 
\begin{tabular}{@{}c@{}} $q=p^k$, $0 < k < m$,\\ $\Gcd{k}{m}=d$, $m/d$ even,\\ $a \in \M \setminus \squares{\M}$.\end{tabular} & & & \\ 
\hline 

$\Family{1}$  & $((1,0,0,B)_q,(0,1,\frac{a}{B},0)_r)$                     &$p^{4l}$ & 
\begin{tabular}{@{}c@{}} $q=p^k,\quad r=p^{k+l}, \quad 0 < k < l$,\\ $ m = 2l, \quad \Gcd{k}{m}=e, \quad m/e \textrm{ odd}$,\\ $a \in \L^\times,\quad B \in \M \setminus \squares{\M}$.\end{tabular} & 
$(p^{e/2},p^{e})$  &  $\geq \left\lfloor\frac{\sigma(n)-1}{2}\right\rfloor \left\lceil\frac{p^l-1}{n}\right\rceil$     & Theorem \ref{thm_family1} \\ 
\hline 

\end{tabular} 
}
\end{center}
\caption{Known infinite families of biprojective commutative semifields of order $p^n$} \label{table_biproj}
\end{table}

The known infinite families of biprojective and other commutative semifields 
and their planar representations are summarized in Tables \ref{table_biproj} and \ref{table_comm}.
Families $\mathcal{A},\mathcal{D},\mathcal{BH}$ reduce to $\mathbb{F}$ when $k \in \{0,m\}$.
Family $\mathcal{ZP}$ reduces to $\mathcal{D}$ when $k = 0$, to $\mathcal{BH}$ when $j = 0$, 
and to $\mathbb{F}$ when $j = k = 0$. Family $\mathcal{S}$ reduces to $\mathcal{ZP}$ when $a = 0$,
and to $\mathcal{D}$ when $k \in \{0,l\}$. We excluded those cases in the Notes 
and also in the Counts columns of Table~\ref{table_biproj}.

\subsection{The Family $\Family{1}$}
The main result of this paper is to prove that 
\begin{itemize}
\item Family $\Family{1}$ gives new commutative semifields, and
\item Family $\Family{1}$ contains an exponential number of non-isotopic 
commutative semifields (in $n$). 
\end{itemize}
An informal way to explain why Family $\Family{1}$ gives such a large number 
of commutative semifields is that their polarizations 
admit only a few $\M$-linear isotopisms (within the family)
due to their complexity ---they contain two non-zero
entries in either component of their planar representations $(1,0,0,a)_q$ and $(0,1,b,0)_r$; 
and the underlying field automorphisms $q$ and $r$ are nontrivial and are not simply 
related to each other. Indeed, our
method in Section \ref{sec_method} will show that $\M$-linear 
isotopisms are essentially the only ones for biprojective semifields
whose autotopism groups satisfy a simply defined condition. The polarization of 
$(0,1,0,0)_q$, which is a component polynomial of many other 
biprojective semifields (except $\mathcal{A}$ and $\mathcal{BH}_{\textrm{even}}$),
admits more such isotopisms and that is the main reason why all $a \in \M$ allowed 
in these constructions lead to isotopic semifields. 
For $\F, \mathcal{D}, \mathcal{BH}_{\textrm{even}}$ and $\mathcal{A}$ 
the reasons for admitting only a small number of non-isotopic semifields include
the simplicity of the defining field automorphisms, e.g., $q \in \{1,p^{m/2}\}$; 
or having the same $(q,q)$ or conjugate $(q,\qbar)$ automorphisms. 
We will explain these in detail in Section \ref{sec_method} (see Theorem~\ref{thm:equivalence}). 
We start by proving that Family $\Family{1}$ indeed gives commutative semifields.


\section{The commutative semifield family $\Family{1}$} \label{sec_families}

The following diagram and its annotations describe our setting. 

\vspace{1em}

\noindent\begin{minipage}{0.35\textwidth}
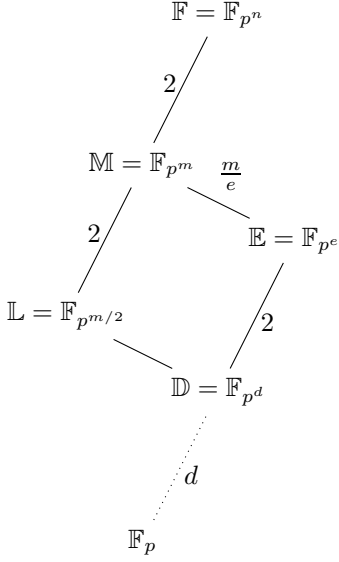
\begin{tikzpicture}
    \node (Q1) at (1,0) {$\f{p}$};
    \node (Q2) at (2,2) {$\D = \f{p^d}$};
    \node (Q3) at (3,4) {$\E = \f{p^e}$};
    \node (Q4) at (0,3) {$\L = \f{p^{m/2}}$};
    \node (Q5) at (1,5) {$\M = \f{p^m}$};
    \node (Q6) at (2,7) {$\F = \f{p^n}$};

    \draw[dotted] (Q1)--(Q2) node[pos=0.7, below, inner sep=0.25cm] {$d$};
    \draw (Q2)--(Q3) node[pos=0.7, below, inner sep=0.25cm] {$2$}; 
    \draw (Q2)--(Q4);
    \draw (Q3)--(Q5) node[pos=0.3, above, inner sep=0.25cm] {$\frac{m}{e}$};
    \draw (Q4)--(Q5) node[pos=0.3, above, inner sep=0.25cm] {$2$};
    \draw (Q5)--(Q6) node[pos=0.3, above, inner sep=0.25cm] {$2$}; 
\end{tikzpicture}
\end{minipage}
\begin{minipage}{0.65\textwidth}
\begin{notation}
\begin{itemize}
\setlength\itemsep{0.3em}
\item[]
\item $p$ is an odd prime.
\item $n = 2m$, $m$ is even.
\item $Q = p^{m/2}$, \quad $Q^2 = p^m$.
\item $q = p^k$, \quad $r = p^{k+m/2} = Qq$ \ \ with $1 \le k \le m-1$. 
\item $e = \Gcd{k}{m}$ with $m/e$ odd. 
\item $d = \Gcd{k+m/2}{m}$.
\item $e = 2d$ by Lemma \ref{lem_thm}.
\item $(\M^\times)^2$ --- the subgroup of non-zero squares in $\M^\times$.
\item $ \L^\times = (\M^\times)^{Q+1} \leq \squares{\M} \leq \M^\times$.
\item $(\M^\times)^{Q-1} \leq \squares{\M} \leq \M^\times$ --- the subgroup of $(Q+1)^\textrm{st}$ roots of unity in $\M^\times$.
\item $\E = \f{q} \cap \M = \f{q^2} \cap \M = \f{r^2} \cap \M$ by Lemma \ref{lem_thm}.
\item $\D = \f{r} \cap \M$.
\end{itemize}
\end{notation}
\end{minipage}


We need the following well known result on the greatest common divisor of $p^i\pm 1$ and $p^m-1$. A proof can be found in \cite{Payne}.

\begin{lemma} \label{lem:gcd}
	Let $i,m \in \N$ and $p$ be a prime. Then
	\begin{itemize}
	\item $\Gcd{p^i-1}{p^m-1} = p^{\Gcd{i}{m}}-1$.
	\item $\Gcd{p^i+1}{p^m-1}=\begin{cases}
		1 & \text{if } m/\Gcd{i}{m} \text{ odd, and } p=2, \\
		2 & \text{if } m/\Gcd{i}{m} \text{ odd, and } p>2, \\
		p^{\Gcd{i}{m}}+1 & \text{if } m/\Gcd{i}{m} \text{ even}.
	\end{cases}$
\end{itemize}
\end{lemma}
	
First, we will prove a lemma.

\begin{lemma}\label{lem_thm}
We have,
\begin{enumerate}[(i)]
\item $(-1) \not\in (\M^\times)^{q-1}$.
\item Any $x \in \squares{\M}$ can be written (twice) as $x = cg$ where 
$c \in \L^\times$ and $g \in (\M^\times)^{Q-1}$. 
\item $\Gcd{k+m/2}{m} = \Gcd{k}{m}/2$.
\item $\E = \f{q} \cap \M = \f{q^2} \cap \M = \f{r^2} \cap \M$.

\end{enumerate}
\end{lemma}
\begin{proof}
\begin{enumerate}[(i)]
\item Recall that $\nu_2(x)$ denotes the index of $2$ in $x$, that is, $\nu_2(x) = h$ if
$2^h|x$ and $2^{h+1}\nmid x$. We have $\nu_2(p^e-1) = \nu_2(Q^2-1)$ since
$Q^2-1 = p^m - 1 = (p^e-1) \sum_{i = 0}^{m/e-1} p^i$ and the fact that an odd number of odd integers
add up to an odd integer. Thus $\nu_2((p^m-1)/2) =\nu_2(p^e-1) - 1$.
The result follows from Lemma \ref{lem:gcd} which shows $\Gcd{q-1}{Q^2-1} = p^e-1$.
\item It is easy to see that $(\M^\times)^{Q-1} \cap \L^\times = \{\pm 1\}$ since $\Gcd{Q-1}{Q+1}=2$
and $xg = yh$ if and only if $(x,g) = (y,h)$ or $(x,g) = (-y,-h)$.
\item Let $k = 2dk'$ and $m=2dm'$ with $\Gcd{k'}{m'} = 1$ and $m'$ odd. Then 
$ \Gcd{2dk'+dm'}{2dm'} = d \cdot \Gcd{2k'+m'}{2m'} = d \cdot \Gcd{2k'+m'}{m'} = d \cdot \Gcd{2k'}{m'} = d$, since $m'+2k'$ is odd.
\item Obvious since $m/e$ is odd and $(qQ)^2 = q^2 \pmod{Q^2-1}$.
\end{enumerate}
\end{proof}

Now we present the family of planar mappings.

\begin{theorem}  \label{thm_family1}
Let $a \in \L^\times$ and $B \in \M^\times \setminus \squares{\M}$ and
let 
\[
F : \M \times \M \to \M \times \M
\]
be defined as
\[
F : (x,y) \mapsto F(x,y) = ((1,0,0,B)_q,(0,1,a/B,0)_r).
\]
Then $F$ is planar.
\end{theorem}
\begin{proof}
We are going to use Lemma \ref{lem_PN}. 
First, 
\begin{align*}
\DE{f}{0}(x,y) &= B(y^q + y)         = 0, \quad \textrm{and},\\
\DE{g}{0}(x,y) &= x^r + \frac{ax}{B} = 0,
\end{align*}
imply $(x,y) = (0,0)$, since otherwise either of
\[
y^{q-1} = -1, \textrm{ and } x^{r-1} = -\frac{a}{B}
\]
lead to a contradiction, since $-1 \not\in (\M^\times)^{q-1}$ by Lemma \ref{lem_thm}, 
and $2|\Gcd{r-1}{Q^2-1}$ and $-a/B \not\in \squares{\M}$. 
Similarly,
\begin{align*}
\DE{f}{\infty}(x,y) &= x^q + x            = 0, \quad \textrm{and},\\
\DE{g}{\infty}(x,y) &= y + \frac{ay^r}{B} = 0,
\end{align*}
has the unique common solution $(x,y)=(0,0)$ with the same argument after changing variables.  
Now for $u \in \M^\times$,
\begin{align*}
\DE{f}{u}(x,y) &= ux^q + u^qx + B(y^q + y) = 0, \quad \textrm{and},\\
\DE{g}{u}(x,y) &= x^r + \frac{a}{B}x + \frac{a}{B} u y^r + u^r y = 0.
\end{align*}
Or,
\begin{align*}
\DE{f}{u}(ux,y) &= u^{q+1}(x^q + x) + B(y^q + y) = 0, \quad \textrm{and},\\
\DE{g}{u}(ux,y) &= u^r (x^r + y) + \frac{a}{B} u (x + y^r) = 0.
\end{align*}
We will proceed to show that $(x,y) = (0,0)$ is the only common solution 
of these equations for $x,y \in \M$. 
Now we can assume $x,y \in \M^\times$, since $x=0$ implies $y=0$ and vice versa for 
$\DE{f}{u}(ux,y) = 0$. 
Furthermore, $x = -y^r$ implies $y - y^{r^2} = 0$ and 
$x,y \in \f{r^2} \cap \M = \f{q^2} \cap \M = \f{q} \cap \M = \E$ by Lemma \ref{lem_thm}. 
Thus $x^q + x = 2x = -2y^r$ and $y^q + y = 2y$, in turn 
\[
2( -u^{q+1}y^r + By) = 0,
\]
or
\[
y^{r-1} = \frac{B}{u^{q+1}}.
\]
This is impossible since $B \not\in \squares{\M}$. The same argument shows
$x^r\ne -y$ and we can concentrate on
\begin{align}
u^{q+1} &= -\frac{B(y^q + y)}{x^q + x}, \textrm{ and}, \label{eqp1}\\
u^{r-1} &= -\frac{a(x + y^r)}{B(x^r+y)},               \label{eqp2}
\end{align}
for $x,y \in \M^\times$ with $x^r \ne -y$ and $x \ne -y^r$.
Now assume \eqref{eqp1} and \eqref{eqp2} hold for such $x,y \in \M^\times$, 
and let
\begin{align}
\phi_q(x,y) &= \frac{y^q + y}{x^q + x} = \frac{\gamma}{cg}, \textrm{ and}, \label{eqp3}\\
\phi_r(x,y) &= \frac{x + y^r}{x^r+y}   = \gamma^Qdh,         \label{eqp4}
\end{align}
for some $c,d \in \L^\times$, $g,h \in (\M^\times)^{Q-1}$ and a fixed $\gamma \in \M^\times \setminus \squares{\M}$,
since $cg$ and $dh$ both run through $\squares{\M}$ independently by Lemma \ref{lem_thm}.
Note that \eqref{eqp1} and \eqref{eqp2} guarantees that $\phi_q(x,y)$ and $\phi_r(x,y)$ are 
non-squares. Multiplying \eqref{eqp3} and \eqref{eqp4} (and also \eqref{eqp1} and \eqref{eqp2}) we get
\[
\gamma^{Q+1} \frac{dh}{cg} = \phi_q(x,y)\phi_r(x,y) = \frac{u^{q+r}}{a} = \frac{u^{q(Q+1)}}{a} \in \L^\times, 
\]
and therefore $\epsilon g = h$ with $\epsilon \in \{\pm 1\}$ since $\gamma^{Q+1}\in \L^\times$. Now let $z = y - x^Q$ and consider
\begin{align*}
 cg(z^q + z + (x^q + x)^Q) &= \gamma(x^q + x), \textrm{ and},\\
 (x^q+x) + z^r             &= \epsilon \gamma^Q dg((x+x^q)^Q + z),
\end{align*}
for $x \in \M^\times$, which is a rewriting of \eqref{eqp3} and \eqref{eqp4}. Or, equivalently,
\begin{align}
 cg(x^q + x)^Q -\gamma (x^q + x) &= -cg(z^q + z), \textrm{ and},\label{eqp5}\\
-\epsilon\gamma^Q dg(x^q + x)^Q + (x^q + x) &= -z^r + \epsilon\gamma^Q dgz.\label{eqp6}
\end{align}
The two new equations generated by
\begin{itemize}
\item Eq. \eqref{eqp5} plus $\gamma$ times Eq. \eqref{eqp6}, and
\item $d\epsilon\gamma^Q$ times Eq. \eqref{eqp5} plus $c$ times Eq. \eqref{eqp6},
\end{itemize}
are as follows:
\begin{align}
 (c-\epsilon\gamma^{Q+1} d) (x^q + x)^Q g &= -\gamma z^r - cgz^q  - (c-\epsilon\gamma^{Q+1} d)gz, \textrm{ and},\label{eqp7}\\
 (c-\epsilon\gamma^{Q+1} d) (x^q + x)     &= -c(z^r - d\epsilon\gamma^Q gz^q).\label{eqp8}
\end{align}
We will show that the common solutions of these equations for $x \in \M^\times$ lead to a
contradiction to our assumption that \eqref{eqp1} and \eqref{eqp2} hold.
Note that $x^q+x \ne 0$ since $x \in \M^\times$ and $-1 \not\in (\M^\times)^{q-1}$.
Now if $z = y - x^Q = 0$, then \eqref{eqp3} becomes
\[
(x^q + x)^{Q-1} = \frac{\gamma}{cg},
\]
which is a contradiction since the left hand side is a square and the right hand side is not.
If $c = \epsilon\gamma^{Q+1} d$ then $z = 0$ by Eq. \eqref{eqp8}, which we have just handled, or 
$z^{r-q} = z^{q(Q-1)} = d\epsilon\gamma^Q g$ by Eq. \eqref{eqp8}, which is again a contradiction
since the right hand side is not a square. 
Then $c \ne \epsilon\gamma^{Q+1}d$. 
Observe that $c-\epsilon\gamma^{Q+1}d \in \L$. Comparing $g$ times Eq.~\eqref{eqp8} to the power $Q$ with Eq.~\eqref{eqp7} yields
\begin{align*}
(-c(z^r - d\epsilon\gamma^Q gz^q))^Q g &= -\gamma z^r - cgz^q  - (c-\epsilon\gamma^{Q+1} d)gz\\
-cgz^q + cd\epsilon\gamma g^{Q+1} z^r  &= -\gamma z^r - cgz^q  - (c-\epsilon\gamma^{Q+1} d)gz\\
         cd\epsilon\gamma   z^r      &= -\gamma z^r - (c-\epsilon\gamma^{Q+1} d)gz\\
       \gamma(cd\epsilon + 1) z^r    &=  - (c-\epsilon\gamma^{Q+1} d)gz.
\end{align*}
Now $cd\epsilon+1= 0$ implies $z = 0$ or $c = \epsilon\gamma^{Q+1}d$, which were handled before. 
Thus, we have
\[
z^{r-1} = -\frac{g}{\gamma}\left(\frac{c-\epsilon\gamma^{Q+1} d}{cd\epsilon+1}\right),
\]
and noting again that $\gamma^{Q+1} \in \L^\times$, we reach another contradiction since the right hand side is not a square.
\end{proof}


\section{A method to determine isotopy of biprojective pre-semifields} \label{sec_method}

There are two usual ways to determine whether two semifields are isotopic. 
The first one is to use isotopy invariants like the nuclei. Since there are 
less than $n^2$ possible configurations for the left/right and central nucleus 
for a commutative pre-semifield on $\F_p^n$, this method is not enough to 
determine whether the number of non-isotopic pre-semifields grows exponentially 
in $n$ or not. The second method works by directly considering all 
possible isotopisms $(N,L,M)$. This is, however, in many cases not feasible unless 
the pre-semifield has a very simple structure. Note that, for commutative semifields, 
some general results were obtained in~\cite{coulterhenderson} that make this approach 
slightly easier and allowed for instance to settle the isotopy question inside the 
family of Zhou-Pott semifields~\cite{ZP13}. However, the semifields in our Family 
$\Family 1$ are more delicate and such a direct approach does not seem 
possible. In this section, we develop a new general technique to determine whether 
two biprojective pre-semifields are isotopic or not. 

\subsection{Group theoretic preliminaries}

We denote the set of all autotopisms of a pre-semifield $\P$ by $\Aut(\P)$. It is 
easy to check that $\Aut(\P)$ is a group under component-wise composition, i.e., 
$(N_1,L_1,M_1) (N_2,L_2,M_2) = (N_1 N_2, L_1 L_2, M_1 M_2)$. 
We view $\Aut(\P)$ as a subgroup of 
$\GL(\F)^3 \cong \GL(\M \times \M)^3 \cong \GL(n,\F_p)^3$.
Our approach is based on the following simple and well-known result.

\begin{lemma} \label{lem:conj_groups}
Let $\P_1=(\F_p^n,+,\ast_1)$, $\P_2=(\F_p^n,+,\ast_2)$ be isotopic
pre-semifields via the isotopism $\gamma \in \GL(\F)^3$. Then 
$\gamma^{-1}\Aut(\P_2)\gamma = \Aut(\P_1)$. 
\end{lemma}
\begin{proof}
Let $\gamma=(N_1,L_1,M_1)\in \GL(\F)^3$ be an isotopism between $\P_1$ 
and $\P_2$ and $\delta = (N_2,L_2,M_2)\in \Aut(\P_2)$. Then 
$\gamma^{-1}\delta\gamma \in \Aut(\P_1)$. Indeed
\begin{align*}
(N_1^{-1} N_2 N_1)(x \ast_1 y ) &= (N_1^{-1}  N_2)(L_1(x) \ast_2 M_1(y)) \\
&= N_1^{-1}((L_2 L_1(x)) \ast_2 (M_2 M_1(y))) \\
&= (L_1^{-1} L_2 L_1(x)) \ast_1 (M_1^{-1}  M_2 M_1(y)),
\end{align*}
so $\gamma^{-1}\Aut(\P_2)\gamma \subseteq \Aut(\P_1)$. The other inclusion follows by symmetry. 
\end{proof}

The central idea of the technique we are going to develop is to identify 
large abelian subgroups (in particular certain Sylow subgroups), in the 
autotopism group of biprojective semifields. We then use tools from group theory 
to obtain strong constraints on when the autotopism groups of two pre-semifields 
are conjugate. This approach is inspired by a similar technique for 
inequivalences of power functions on finite fields developed by 
Dempwolff~\cite{DempwolffPower} and Yoshiara~\cite{YoshiaraPower}.\\

First we recall the well-known Sylow Theorems (see for instance \cite[Chapter 4]{Hall}).

\begin{theorem}[Sylow Theorems]
Let $G$ be a group with order $p^m s$, with $p$ prime, $m > 0$ and $p\nmid s$. Then,
\begin{enumerate}[(i)]
\item $G$ has a subgroup of order $p^m$, called a Sylow $p$-subgroup of $G$.
\item Every $p$-subgroup of $G$ is contained in a Sylow $p$-subgroup of $G$.
\item The Sylow $p$-subgroups of $G$ are conjugate in $G$.
\end{enumerate}
\end{theorem}

We will identify Sylow $p'$-subgroups of $\Aut(\P)$ when $\P$ is a biprojective
pre-semifield. In order to do that, we need to find a suitable prime $p'$, for which
we will employ Zsigmondy's Theorem (see for instance {\cite[Chapter IX., Theorem 8.3]{HuppertII}}). 

\begin{theorem}[Zsigmondy's Theorem]
For every prime $p$ and $m > 2$ except when $(p,m)=(2,6)$, there exists a 
\textbf{$p$-primitive divisor} $p'$ of $p^m-1$, that is, $p'$ prime, $p'|p^m-1$ 
and $p'\nmid p^i-1$ for all $1 \le i \le m-1$.
\end{theorem}

We write $\f{p}$-linear mappings $L\in \End(\F)$ from $\F$ to itself as 
$2 \times 2$ matrices of $\f{p}$-linear mappings  in $\End(\M)$.
That is,
\[
L=\begin{pmatrix}
	L_1 & L_2 \\
	L_3 & L_4
\end{pmatrix}, \textrm{ for } L_i \in \End(\M). 
\]
Set 
\[
\gamma_a=(N_a,L_a,M_a)\in \GL(\F)^3 \quad \textrm{ with } 
N_a=\begin{pmatrix}
	m_{a^{q+1}} & 0 \\
	0 & m_{a^{r+1}}
\end{pmatrix}, \quad 
L_a=M_a=\begin{pmatrix}
	m_{a} & 0 \\
	0 & m_{a}
\end{pmatrix}, 
\]
where $m_a$ denotes  multiplication with the finite field element $a \in \M^\times$. 
For simplicity, we write these diagonal matrices also in the form $\diag(m_{a},m_{a})$, 
so 
\[
\gamma_a = (\diag(m_{a^{q+1}},m_{a^{r+1}}),\diag(m_{a},m_{a}),\diag(m_{a},m_{a})).
\]
An important fact which follows immediately from biprojectivity is that 
$\gamma_a \in \Aut(\P)$ for all $a \in \M^\times$ when $\P$ is a $(q,r)$-biprojective 
pre-semifield, which can be readily verified using Eq.~\eqref{eq_polarization}.
We fix some further notation that we will use from now on:
\begin{notation}
\begin{itemize}
\item[]
\item Set $q=p^k$ and $r=p^l$.
\item Set $\overline{q}=p^{m-k}$ and $\overline{r}=p^{m-l}$, that is $q\overline{q}\equiv r\overline{r} \equiv 1 \pmod{p^m-1}$.
\item Define the cyclic group 
\[
	Z^{(q,r)} = \{ \gamma_a : a \in \M^\times\},
\] of order $p^m-1$. 
\item Let $p'$ be a $p$-primitive divisor of $p^m-1$. Such a prime $p'$ always exists if 
$m > 2$ and $(p,m) \neq (2,6)$ by Zsigmondy's Theorem. In our case, we have $p > 2$. We will 
also stipulate $m > 2$. Note that $p' \neq 2$ since $p'\nmid p-1$.
\item Let $R$ be the unique Sylow $p'$-subgroup of $\M^\times$.
\item Define
\[
Z_R^{(q,r)} = \{ \gamma_a \colon a \in R\},
\]
which is the unique Sylow $p'$-subgroup of $Z^{(q,r)}$ with $|R|$ elements. 
\item For a $(q,r)$-biprojective pre-semifield $\P$, denote by 
\[
C_{\P} = C_{\Aut(\P)}(Z_R^{(q,r)}), 
\]
the centralizer of $Z_R^{(q,r)}$ in $\Aut(\P)$.
\item Define
\[
	S = \{\diag(m_a,m_a) \colon a \in \M^\times\},
\]
and
\[
	S_R = \{\diag(m_a,m_a) \colon a \in R\}.
\]
\end{itemize}
\end{notation}

We start by identifying a subgroup of the autotopism group of any 
$(q,r)$-biprojective pre-semifield. The following lemma is straightforward, 
but very important for our paper.
\begin{lemma}\label{lem_first}
Let $\P$ be any $(q,r)$-biprojective pre-semifield. Then 
\[
Z_R^{(q,r)} \le Z^{(q,r)} \le \Aut(\P).
\]
\end{lemma}
\begin{proof}
Follows directly from Eq.~\eqref{eq_polarization}.
\end{proof}

We continue with a simple observation on $R$.
\begin{lemma} \label{lem:irredm}
	Let $a\in R$, $a \neq 1$. Then $a$ is not contained in a proper subfield of $\M$.
\end{lemma}
\begin{proof}
	Clearly, $a$ is contained in the subfield $\F_{p^i}$ if and only if $a^{p^i-1}=1$, i.e. the multiplicative order of $a$ is a divisor of $p^i-1$. The order of $a$ is a power of $p'$ and $p'$ does not divide $p^i-1$ for any $i<m$, so $a$ is not contained in a proper subfield of $\F_{p^m}$.
\end{proof}

Now we observe that the normalizer and the centralizer of certain subgroups
of $\GL(\F)$ must have certain shape.

\begin{lemma} \label{lem:yoshiara}
 Let $N_{\GL(\F)}(S_R)$, $N_{\GL(\F)}(S)$ and $C_{\GL(\F)}(S_R)$, $C_{\GL(\F)}(S)$ be the normalizers and the centralizers of 
 $S_R$ and $S$ in $\GL(\F)$. Then
	\begin{enumerate}[(a)]
		\item $N_{\GL(\F)}(S_R)=N_{\GL(\F)}(S) = \left\{\begin{pmatrix}
			m_{c_1} \tau & m_{c_2} \tau \\
			m_{c_3} \tau&m_{c_4} \tau
		\end{pmatrix} \colon c_1,c_2,c_3,c_4 \in \M, \tau \in \Gal(\M/\F_p)\right\}\cap \GL(\F)$,
		\item $C_{\GL(\F)}(S_R)=C_{\GL(\F)}(S) = \left\{\begin{pmatrix}
			m_{c_1}  & m_{c_2}  \\
			m_{c_3}&m_{c_4} 
		\end{pmatrix} \colon c_1,c_2,c_3,c_4 \in \M \right\}\cap \GL(\F)$.
	\end{enumerate}
\end{lemma}
 \begin{proof}
	We present the proof only for the more delicate case $S_R$. The proof for $S$ is identical with $\M^\times$ substituting $R$ throughout.

	Let $\begin{pmatrix}
		A_1 & A_2 \\
		A_3 & A_4
	\end{pmatrix} \in N_{\GL(\F)}(S_R) $, where $A_1,A_2,A_3,A_4 \in \End(\M)$. Then  
	\[\begin{pmatrix}
		A_1 & A_2 \\
		A_3 & A_4
	\end{pmatrix} \diag(m_a,m_a) = \diag(m_b,m_b)\begin{pmatrix}
		A_1 & A_2 \\
		A_3 & A_4
	\end{pmatrix}, \]
	for all $a \in R$ and some $b=\pi(a)$ where $\pi \colon R \rightarrow R$ is a bijection. 
	Simple matrix multiplication implies $A_i(ax)=bA_i(x)$ for $i \in \{1,2,3,4\}$.

   We now write the mappings as linearized polynomials, i.e. $A_i=\sum_{j=0}^{m-1} c_{j,i}x^{p^j}$ for $i \in \{1,2,3,4\}$. The equations above then immediately yield
	\[\sum_{j=0}^{m-1} c_{j,i}a^{p^j}x^{p^j} = b\sum_{j=0}^{m-1} c_{j,i}x^{p^j}\]
	for $i \in \{1,2,3,4\}$ and all $a \in R$.
	We now compare the coefficients of these polynomials. If $a\neq 1$, it is not contained in any proper subfields of $\M$ by Lemma~\ref{lem:irredm}, so we have $b=a^{p^j}$ only for at most one $j \in \{0,1,\dots,m-1\}$. So $A_1,A_2,A_3,A_4$ are zero or monomials of the same degree $p^j$, which proves the statement for the normalizer. In the case of the centralizer, we have $b=a$, so we get the same possible mappings except with $j=0$ forced.
	\end{proof}

We have shown in Lemma \ref{lem_first} that $Z_R^{(q,r)}$ is a subgroup of $\Aut(\P)$.
Now we show that, under a certain condition which is key to our proofs, it is not just a
Sylow $p'$-subgroup of $Z^{(q,r)}$, but even a Sylow $p'$-subgroup of $\Aut(\P)$.

\begin{lemma}\label{lem:zpsylow}
Let $\P$  be a $(q,r)$-biprojective pre-semifield. Assume that 
$C_{\P}$ contains $Z^{(q,r)}$ as an index $I$ subgroup such that $p'$ does not divide $I$. 
Then $Z_R^{(q,r)}$ is a Sylow $p'$-subgroup of $\Aut(\P)$.
\end{lemma}
\noindent\begin{minipage}{0.85\textwidth}
\begin{proof}
First define 
\[
	U = \{(\diag(m_a,m_b),\diag(m_c,m_d),\diag(m_e,m_f)) \colon a,b,c,d,e,f \in R\}. 
\]
Clearly, $|U|=|R|^6$. We will now show that $U$ is a Sylow $p'$-subgroup of $\GL(\F)^3$. 
We have 
\[
	|\GL(\F)| = |\GL(2m,\F_p)| = p^{m(2m-1)} \prod_{i=1}^{2m}(p^i-1). 
\]
Clearly, $p^{m+i}-1 \equiv p^i-1 \pmod{p^m-1}$. As $p'$ is a $p$-primitive divisor 
of $p^m-1$, all integers $p^j-1$ with $j$ in $[1,2m]$ are coprime to $p'$, except 
for $j \in \{m,2m\}$. Furthermore, the $p'$-part of $p^{2m}-1 = (p^{m}-1)(p^m+1)$ 
is $|R|$ since $p'\neq 2$. Thus the $p'$-part of 
$|\GL(\F)|$ is $|R|^2$ and $U$ is a Sylow $p'$-subgroup of $\GL(\F)^3$ as claimed. 
All Sylow $p'$-subgroups of $\GL(\F)^3$ are abelian since $U$ is abelian,
by Sylow Theorem (iii).	
Note that any $p'$-subgroup of a group $G$ is contained in a Sylow $p'$-subgroup 
of $G$ by Sylow Theorem (ii).
Let $T$ be a Sylow $p'$-subgroup of $\Aut(\P)$ that contains the 
$p'$-group $Z_R^{(q,r)}$. Then $T$ itself is (again by Sylow Theorem (ii)) contained in
a Sylow $p'$-subgroup of $\GL(\F)^3$, say $U'$. In particular, $T$ is abelian. 
This implies that $T$ is a subgroup of the centralizer 
$C_{\P}$ of $Z_R^{(q,r)}$ in $\Aut(\P)$. By assumption, $Z^{(q,r)}$ is an index 
$I$ subgroup of $C_{\P}$ and $p'$ does not divide $I$.
Moreover $Z_R^{(q,r)}$ is a Sylow $p'$-subgroup of $Z^{(q,r)}$ and therefore
$p'\nmid [Z^{(q,r)}:Z_R^{(q,r)}] = I_1$. Let $[T:Z_R^{(q,r)}] = I_2 = p'^h$ for $h \ge 0$, 
since both are $p'$-groups. Since $I_2|I_1 I$,
and $p'\nmid I_1 I$, we must have $p'\nmid I_2$ and $I_2 = 1$. Thus,
$Z_R^{(q,r)}=T$ and $Z_R^{(q,r)}$ is a Sylow $p'$-subgroup of $\Aut(\P)$.
\end{proof}
\end{minipage}
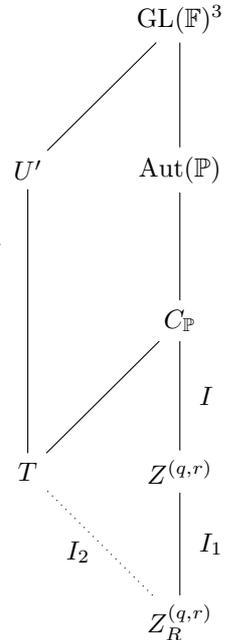
\begin{minipage}{0.15\textwidth}
\begin{tikzpicture}
    \node (Q1) at (2,0) {$Z_R^{(q,r)}$};
    \node (Q2) at (2,2) {$Z^{(q,r)}$};
    \node (Q3) at (2,4) {$C_\P$};
    \node (Q4) at (2,6) {$\Aut(\P)$};
    \node (Q5) at (0,2) {$T$};
    \node (Q6) at (0,6) {$U'$};
    \node (Q7) at (2,8) {$\GL(\F)^3$};

    \draw (Q1)--(Q2) node[pos=0.5, right, inner sep=0.25cm] {$I_1$}; ;
    \draw (Q2)--(Q3) node[pos=0.5, right, inner sep=0.25cm] {$I$}; 
    \draw (Q3)--(Q4);
    \draw (Q3)--(Q5);
    \draw[dotted] (Q1)--(Q5) node[pos=0.7, below, inner sep=0.25cm] {$I_2$};
    \draw (Q4)--(Q7); 
    \draw (Q6)--(Q7); 
    \draw (Q5)--(Q6); 
\end{tikzpicture}
\end{minipage}

\subsection{A theorem on isotopisms between biprojective semifields}

Now we are going to show that if two biprojective pre-semifields are isotopic, an isotopism $(N,L,M)$ 
that satisfies strong requirements on the shape of the linearized polynomials $N,L$ and $M$ has to exist,
whenever the condition appearing in the assumption of Lemma \ref{lem:zpsylow} is satisfied. 
We will name it Condition \eqref{eq:condition}. First we need a simple lemma.

\begin{lemma} \label{lem:a0condition}
Let $(x,y) \mapsto F(x,y) = (f(x,y),g(x,y)) = ((a_0,b_0,c_0,d_0)_q,(a_1,b_1,c_1,d_1)_r)$ 
be a $(q,r)$-biprojective mapping for arbitrary values of $q,r$. If $F$ is planar, 
then $(a_0,a_1)\neq (0,0)$ and $(d_0,d_1)\neq (0,0)$. That is to say, all biprojective 
pre-semifield polarizations $\Delta_F((x,y),(u,v))$ have a component 
that contains both monomials $x^\sigma u$ and $xu^\sigma$ and 
a component (not necessarily different) that contains both $y^\tau v$ and $yv^\tau$ for 
$\sigma,\tau \in \{q,r\}$ depending on the component. 
\end{lemma}
\begin{proof}
If $a_0 = a_1 = 0$, then $\DE{f}{\infty}(x,y)=0=\DE{g}{\infty}(x,y)$ for $y=0$ and arbitrary $x$, contradicting Lemma~\ref{lem_PN} for $u=\infty$. The contradiction for $d_0,d_1$ is obtained similarly by considering Lemma~\ref{lem_PN} for $u=0$. The statement on the corresponding pre-semifield follows immediately by Eq.~\eqref{eq_polarization} and the fact that $(a_0,a_1)\neq (0,0)$.
\end{proof}

We are now ready to prove the main result of this section: If two biprojective 
semifields are isotopic, then there exists an isotopism between them of a very specific form.

\begin{theorem} \label{thm:equivalence}
Let $\P_1=(\M \times \M,+,\ast_1)$ and $\P_2=(\M \times \M,+,\ast_2)$ be $(q_1,r_1)$- and  
$(q_2,r_2)$-biprojective pre-semifields, respectively, such that 
$q_1 \not\in \{ r_1, \overline{r_1} \}$, $1 \notin \{q_1,r_1\}$ and $Q \not\in \{q_1,r_1\}$, where
$q_i = p^{k_i}$ and $r_i = p^{l_i}$ for $i \in \{1,2\}$.
Assume that 
	\begin{equation}
		C_{\P_1} \text{ contains }Z^{(q_1,r_1)} \text{ as an index } I \text{ subgroup such that }p'\text{ does not divide }I.
	\tag{C}\label{eq:condition}
	\end{equation}
If $\P_1,\P_2$ are 
 isotopic, then there exists an 
 isotopism $\gamma=(N,L,M) \in \GL(\F)^3$, with the following properties:
\begin{itemize}
	\item All non-zero subfunctions of $N,L$ and $M$ are monomials.
	\item All non-zero subfunctions of $L$ and $M$ have the same degree $p^t$.
	\item We have,
	\begin{itemize} 
		\item either $k_1 \equiv \pm k_2 \pmod{m}$ and $l_1 \equiv \pm l_2 \pmod{m}$,
		\item or $k_1 \equiv \pm l_2 \pmod{m}$ and $l_1 \equiv \pm k_2 \pmod{m}$.
	\end{itemize}
\end{itemize}
More precisely, we have either,
\begin{itemize}
\item $N_2=N_3=0$ and $N_1,N_4 \neq 0$,
\item $k_1 \equiv \pm k_2 \pmod{m}$ and $l_1 \equiv \pm l_2 \pmod{m}$,
\item if $k_1 \equiv  k_2 \pmod{m}$ (resp. $l_1 \equiv  l_2 \pmod{m}$) then $N_1$ (resp. $N_4$) is a monomial of degree $p^t$,
\item if $k_1 \equiv  -k_2 \pmod{m}$ (resp. $l_1 \equiv  -l_2 \pmod{m}$) then $N_1$ (resp. $N_4$) is a monomial of degree $p^{t+k_2}$ (resp. $p^{t+l_2}$),
\end{itemize}
or,
\begin{itemize}
\item $N_1=N_4=0$ and $N_2,N_3 \neq 0$,
\item $k_1 \equiv \pm l_2 \pmod{m}$ and $l_1 \equiv \pm k_2 \pmod{m}$,
\item if $k_1 \equiv  l_2 \pmod{m}$ (resp. $l_1 \equiv  k_2 \pmod{m}$) then $N_3$ (resp. $N_2$) is a monomial of degree $p^t$,
\item if $k_1 \equiv  -l_2 \pmod{m}$ (resp. $l_1 \equiv  -k_2 \pmod{m}$) then $N_3$ (resp. $N_2$) is a monomial of degree $p^{t+l_2}$ (resp. $p^{t+k_2}$).
\end{itemize}
\end{theorem}
\begin{proof}
	
Set 
\begin{align*}
	(x,y)\ast_1(u,v)&=(f_1(x,y,u,v),g_1(x,y,u,v)), \textrm{ and}\\
	(x,y)\ast_2(u,v)&=(f_2(x,y,u,v),g_2(x,y,u,v)).
\end{align*}
By Lemma \ref{lem_first}, we have $Z_R^{(q_1,r_1)} \le \Aut(\P_1)$
and $Z_R^{(q_2,r_2)} \le \Aut(\P_2)$.
Assume $\P_1$ and $\P_2$ are isotopic via the isotopism $\delta \in \GL(\F)^3$ 
that maps $\P_1$ to $\P_2$. 
Then $\delta^{-1} \Aut(\P_2) \delta = \Aut(\P_1)$ by Lemma~\ref{lem:conj_groups}. 
Observe that $|\delta^{-1}Z_R^{(q_2,r_2)} \delta|=|R|=|Z_R^{(q_1,r_1)}|$, 
so $Z_R^{(q_1,r_1)}$ and $\delta^{-1}Z_R^{(q_2,r_2)}\delta$ are 
Sylow $p'$-subgroups of  $\Aut(\P_1)$ by Lemma~\ref{lem:zpsylow} as long as 
Condition~\eqref{eq:condition} holds. In particular, these two subgroups are 
conjugate in $\Aut(\P_1)$ by Sylow Theorem (iii), i.e., there exists a 
$\lambda \in \Aut(\P_1)$ such that 
\begin{equation}
\label{eq:conjugated}
\lambda^{-1} \delta^{-1} Z_R^{(q_2,r_2)} \delta \lambda = (\delta \lambda)^{-1} Z_R^{(q_2,r_2)} (\delta\lambda) = Z_R^{(q_1,r_1)}.
\end{equation}

Set $\gamma= (N,L,M)=\delta \lambda$. Note that $\gamma$ is an isotopism between 
$\P_1$ and $\P_2$ since $\lambda \in \Aut(\P_1)$. Eq.~\eqref{eq:conjugated} then 
immediately implies that 
\begin{align*}
\diag(m_{a^{q_2+1}},m_{a^{r_2+1}}) N &= N\diag(m_{b^{q_1+1}},m_{b^{r_1+1}}) \\
\diag(m_{a},m_{a}) L &= L \diag(m_{b},m_{b}) \\
\diag(m_{a},m_{a}) M &= M \diag(m_{b},m_{b}),
\end{align*}
for all $a \in R$ and $b=\pi(a)$ where $\pi \colon R \rightarrow R$ is a permutation. 
In particular, $L$ and $M$ are in the normalizer of 
$S_R = \{\diag(m_{a},m_{a}) \colon a \in R\}$. 
By Lemma~\ref{lem:yoshiara}, all of the four subfunctions of $L$ and $M$  are zero or 
monomials of the same degree, say $p^{t_2}$ and $p^{t_3}$, respectively. Then, for all $(x,y),(u,v)\in\M^2$,
\begin{align*}
L(x,y) \ast_2 M(u,v) &= (a_2x^{p^{t_2}}+b_2y^{p^{t_2}},c_2x^{p^{t_2}}+d_2y^{p^{t_2}}) \ast_2 
                        (a_3u^{p^{t_3}}+b_3v^{p^{t_3}},c_3u^{p^{t_3}}+d_3v^{p^{t_3}}),\\
	                   &= (h_1(x,y,u,v), h_2(x,y,u,v))
\end{align*}
for some $a_2,b_2,c_2,d_2,a_3,b_3,c_3,d_3 \in \M$.
We also have
	\begin{align*}
		N((x,y)\ast_1(u,v)) = (&N_{1}(f_1(x,y,u,v))+N_{2}(g_1(x,y,u,v)),\\
		                       &N_{3}(f_1(x,y,u,v))+N_{4}(g_1(x,y,u,v))).
	\end{align*}
Let us now assume that $N((x,y)\ast_1(u,v)) = L(x,y)\ast_2 M(u,v)$. We consider the first component, i.e. $h_1(x,y,u,v) = N_{1}(f_1(x,y,u,v))+N_{2}(g_1(x,y,u,v))$. Lemma~\ref{lem:a0condition} implies both monomials $x^qu$ and $xu^q$ occur in one of the two components of $\P_1$. Since switching the components clearly preserves isotopy, we can assume without loss of generality that they occur in the first component $f_1$. Let us for now assume $N_1 \neq 0$. We know that $N_1(f_1(x,y,u,v))$ has then terms of the form 
	\begin{equation}
	x^{p^{k_1+t}}u^{p^t} \text{ and } x^{p^t}u^{p^{k_1+t}}
	\label{eq:degrees1}
	\end{equation}
	 for at least one $0 \leq t \leq m-1$. Observe that the differences of 
	 $p$-adic valuations of exponents in the $x$- and $u$-degrees of the 
	 monomials are $k_1 + t - t = k_1$ and $t-k_1-t = -k_1$, respectively. In particular, 
	 if $q_1 \neq r_1$, $q_1 \neq \overline{r_1}$, then these terms cannot be 
	 canceled out by $N_2(g_1(x,y,u,v))$. Since $\P_2$ is a $(q_2,r_2)$-biprojective 
	 pre-semifield, all possible terms in $h_1$ are of the form 
		\begin{equation}
		w^{p^{k_2+t_2}}z^{p^{t_3}} \text{ or }w^{p^{t_2}}z^{p^{k_2+t_3}},
	\label{eq:degrees2}
	\end{equation}
	 where $w \in \{x,y\}$, $z \in \{u,v\}$ and $0\leq t_2,t_3 \leq m-1$. 
	 Comparing Eqs. ~\eqref{eq:degrees1} and~\eqref{eq:degrees2} gives either

\noindent\begin{minipage}{0.5\textwidth}
\begin{align*}
k_1+t &\equiv k_2+t_2 \pmod m, \\
t     &\equiv t_3 \pmod m, \\
t     &\equiv t_2 \pmod m, \\
k_1+t &\equiv k_2+t_3 \pmod m, \textrm{ or,}
\end{align*}
\end{minipage}
\begin{minipage}{0.5\textwidth}
\begin{align*}
k_1+t &\equiv t_2 \pmod m, \\
t     &\equiv k_2+t_3 \pmod m, \\
t     &\equiv k_2+t_2 \pmod m, \\
k_1+t &\equiv t_3 \pmod m.
\end{align*}
\end{minipage}
	The first possibility is equivalent to $t=t_2=t_3$, $k_1 \equiv k_2 \pmod m$ and the second implies $t_2=t_3$, $t=t_2+k_2$ and $k_1 \equiv -k_2 \pmod m$. Note in particular that in any case $t_2=t_3$. Moreover, both cases cannot occur simultaneously since, by assumption,  $k_1 \not\equiv m/2 \pmod m$. We conclude that $N_1$ is a monomial with the same degree $p^{t_2}$ as the $L_i,M_i$ if $k_1 \equiv k_2 \pmod m$, and a monomial with the degree $p^{t_2+k_2}$ if $k_1 \equiv -k_2 \pmod m$. 
\vskip1em

	Now assume $N_2\neq 0$ and observe that the terms of $N_2(g_1(x,y,u,v))$ are of the form 
	\begin{equation}
		w^{p^{l_1+t}}z^{p^t} \text{ or }w^{p^t}z^{p^{l_1+t}}
	\label{eq:n2}
	\end{equation},
where $w \in \{x,y\}$, $z \in \{u,v\}$ and some $0 \leq t \leq m-1$. In particular, the difference of  $p$-adic valuations of exponents of the two monomials is $l_1$ or $-l_1$. This however yields a contradiction since the difference of $p$-adic valuations of exponents  in Eq.~\eqref{eq:degrees2} is $k_2$ or $-k_2$, that is by the considerations above, $-k_1$ or $k_1$, which leads to $k_1\equiv \pm l_1 \pmod m$ which is not possible since $q_1 \neq r_1$, $q_1 \neq \overline{r_1}$. We conclude that $N_2=0$.  
	
		Let us now consider the second component. Since $N_2=0$, we must have $N_4\neq 0$ since $N$ is bijective. The terms of $N_4(g_1(x,y,u,v))$ are then again of the same form as in Eq.~\eqref{eq:n2}. Similar to Eq.~\eqref{eq:degrees2}, the terms in $h_2$ are of the form (using $t_2=t_3$)
				\begin{equation}
		w^{p^{l_2+t_2}}z^{p^{t_2}} \text{ or }w^{p^{t_2}}z^{p^{l_2+t_2}},
	\label{eq:degrees3}
	\end{equation}
			where $w \in \{x,y\}$, $z \in \{u,v\}$.
	A comparison between the Eqs.~\eqref{eq:n2} and~\eqref{eq:degrees3} yields either $t=t_2$ and $l_1	\equiv l_2 \pmod m$ or $t=l_2+t_2$ and $l_1 \equiv -l_2 \pmod m$. Again, this means that $N_4$ is a monomial of degree $p^{t_2}$ or $p^{l_2+t_2}$ as both cases cannot occur simultaneously since $l_1 \not\equiv m/2 \pmod m$. 
	
	Since $q_1 \neq r_1$, $q_1 \neq \overline{r_1}$, we can again deduce $N_3=0$ with the same argument we used to prove $N_2=0$ before. This concludes the case $N_1\neq 0$.
	
	Now assume $N_1=0$. Since $N$ is bijective, this implies $N_3\neq 0$. We can then employ 
	the entire argument, just starting with the second component and exchanging $k_2$ and $l_2$, 
	$N_1$ and $N_3$, and $N_2$ and $N_4$ throughout. We conclude that in this case $N_3,N_2\neq 0$ 
	and $N_1=N_4=0$.

	\end{proof}
\begin{remark}
\begin{enumerate}[(i)]
\item 
We exclude the cases $q_1=r_1$, $q_1=\overline{r_1}$, $1\in \{q_1,r_1\}$, and $Q\in\{q_1,r_1\}$.
It is possible to give (slightly weaker) versions of Theorem~\ref{thm:equivalence} also 
in the excluded cases. We avoided these cases to simplify the exposition.
For instance, when we allow $Q \in \{q_1,r_1\}$, then the subfunctions of $N$ may be 
binomials of the form $N_i=ax^{p^t}+bx^{p^{t+m/2}}$. We will showcase an isotopy of 
this kind in Remark~\ref{rem:q1} in Section~\ref{sec_equiv}. Observe that in the version 
we have given, all non-zero subfunctions of $N,L,M$ are monomials. We chose this 
presentation of the theorem to avoid listing unnecessary special cases that we do 
not need in this paper.

\item We will mainly use Theorem~\ref{thm:equivalence} to determine the number of 
isotopy classes in Family $\Family 1$. Of course, it can also be used to give 
alternative (in most cases simpler) proofs of the number of isotopy classes of the 
known commutative biprojective pre-semifields.
\end{enumerate}

\end{remark}

Theorem~\ref{thm:equivalence} enables us to settle the isotopy question 
of biprojective pre-semifields with relative ease as long as Condition~\eqref{eq:condition} 
is satisfied. 
In the next section we will first show that Condition~\eqref{eq:condition}
is satisfied for our family and then use Theorem~\ref{thm:equivalence} to
determine the number of non-isotopic semifields in the family. Note that the 
condition $m>2$ we stipulate in this section is not restrictive when considering 
$\Family 1$ since it does not yield semifields when $m$ is a power of $2$.


\section{Isotopisms within the Family $\mathcal{S}$} \label{sec_counts1}

In this section we will show that the number of non-isotopic semifields
within the Family $\Family{1}$ is exponential in $n$ (when $n=4t$, where
$t$ is not a power of $2$). 
We first need to check Condition~\eqref{eq:condition} in 
Theorem~\ref{thm:equivalence} for the pre-semifields in the Family $\mathcal{S}$.
The following lemma does that in a straightforward manner.

\begin{lemma} \label{lem:centralizer_F1}
Let $n = 2m$ and $\P=(\M \times \M,+,\ast)$ be a $(q,r)$-biprojective pre-semifield 
in the Family $\mathcal{S}$. Then 
\begin{align*}
|C_{\P}| &=  (p^m-1)(p^{\Gcd{k}{m}}-1), \textrm{ or}\\
|C_{\P}| &= 2(p^m-1)(p^{\Gcd{k}{m}}-1).
\end{align*}
In particular, Condition~\eqref{eq:condition} is always satisfied.
\end{lemma}
\begin{proof}
	If $(N,L,M) \in C_{\P}$ then, by Lemma~\ref{lem:yoshiara}, the subfunctions $L_i$ and $M_i$ for $i \in \{1,2,3,4\}$ are zero or monomials of degree $1$, we write 
\begin{align*}
& L_{1}(x)=a_2x,\quad  L_{2}(x)=b_2x,\quad  L_{3}(x)=c_2x,\quad  L_{4}(x)=d_2x,\\
& M_{1}(x)=a_3x,\quad  M_{2}(x)=b_3x,\quad  M_{3}(x)=c_3x,\quad  M_{4}(x)=d_3x.
\end{align*}
	We then have 
\begin{align}
	L(x,y)\ast M(u,v) =&(a_2x+b_2y,c_2x+d_2y)\ast(a_3u+b_3v,c_3u+d_3v)\nonumber\\
	                  =&((a_2x+b_2y)^q(a_3u+b_3v)+ (a_2x+b_2y)(a_3u+b_3v)^q\nonumber \\
	                   &+B((c_2x+d_2y)^q(c_3u+d_3v)+(c_2x+d_2y)(c_3u+d_3v)^q),\nonumber \\
	                   &(a_2x+b_2y)^{qQ}(c_3u+d_3v)+(c_2x+d_2y)(a_3u+b_3v)^{qQ}\nonumber \\
	&+\frac{a}{B}((a_2x+b_2y)(c_3u+d_3v)^{qQ}+(c_2x+d_2y)^{qQ}(a_3u+b_3v))).\label{eq:F1_RHS}
\end{align}
	Similarly, we have
	\begin{align*}
		N((x,y)\ast(u,v)) = (&N_{1}(x^qu+xu^q+B(y^qv+yv^q))+N_{2}(x^{qQ}v+yu^{qQ}+\frac{a}{B}(xv^{qQ}+y^{qQ}u)),\\
		&N_{3}(x^qu+xu^q+B(y^qv+yv^q))+N_{4}(x^{qQ}v+yu^{qQ}+\frac{a}{B}(xv^{qQ}+y^{qQ}u))).
	\end{align*}
	By comparing the degrees, it is then easy to see that $N((x,y)\ast(u,v))=L(x,y)\ast M(u,v)$ implies $N_2=N_3=0$ and $N_1=a_1x$, $N_4=d_1x$ for some $a_1,d_1 \in \M^{\times}$. Then
	\begin{equation}
		N((x,y) \ast (u,v))  = (a_1(x^qu+xu^q+B(y^qv+yv^q)), d_1(x^{qQ}v+yu^{qQ}+\frac{a}{B}(xv^{qQ}+y^{qQ}u))).
	\label{eq:F1_LHS}
	\end{equation}

	We compare the coefficients of
	$x^qu$, $xu^q$, $x^qv$, $xv^q$, $y^qu$, $yu^q$, $y^qv$, $yv^q$ in the first 
	component of Eqs.~\eqref{eq:F1_RHS} and \eqref{eq:F1_LHS} to get the following
	$8$ equations:		
		
 \noindent\begin{minipage}{0.5\textwidth}
\begin{align}
	a_1 &= a_2^qa_3+Bc_2^qc_3 \label{eq:F4_l1}\\
	a_1&=a_2a_3^q+Bc_2c_3^q \label{eq:F4_l2}\\	
	0 &= a_2^qb_3+Bc_2^qd_3\label{eq:F4_l3}\\
	0 &= a_2b_3^q+Bc_2d_3^q\label{eq:F4_l4} 
\end{align}
    \end{minipage}%
    \begin{minipage}{0.5\textwidth}
	\begin{align}
	0 &= b_2^qa_3+Bd_2^qc_3\label{eq:F4_r1} \\
	0 &= b_2a_3^q+Bd_2c_3^q\label{eq:F4_r2}\\
	Ba_1 &= b_2^qb_3+Bd_2^qd_3 \label{eq:F4_r3}\\
	Ba_1 &= b_2b_3^q+Bd_2d_3^q.\label{eq:F4_r4}
	\end{align}
    \end{minipage}\vskip1em
	And similarly the $8$ equations that come from comparing the coefficients in the second component:
	
 \noindent\begin{minipage}{0.5\textwidth}
\begin{align}
	0 &= a_2^{qQ}c_3+(a/B)c_2^{qQ}a_3 \label{eq:F4_l12}\\
	0&=(a/B)a_2c_3^{qQ}+c_2a_3^{qQ}\label{eq:F4_l22}\\	
	d_1 &= a_2^{qQ}d_3+(a/B)c_2^{qQ}b_3\label{eq:F4_l32}\\
	d_1a/B &= (a/B)a_2d_3^{qQ}+c_2b_3^{qQ}\label{eq:F4_l42} 
\end{align}
    \end{minipage}%
    \begin{minipage}{0.5\textwidth}
	\begin{align}
	d_1a/B &= b_2^{qQ}c_3+(a/B)d_2^{qQ}a_3\label{eq:F4_r12} \\
	d_1 &= (a/B)b_2c_3^{qQ}+d_2a_3^{qQ}\label{eq:F4_r22}\\
	0 &= b_2^{qQ}d_3+(a/B)d_2^{qQ}b_3 \label{eq:F4_r32}\\
	0 &= (a/B) b_2d_3^{qQ}+d_2b_3^{qQ}.\label{eq:F4_r42}
	\end{align}
    \end{minipage}\vskip1em
		
	Let us first assume that $a_2,c_2,a_3,c_3 \neq 0$. Then by Eqs.~\eqref{eq:F4_l12} and~\eqref{eq:F4_l22}, we have 
	\[ - \frac{a}{B} = \frac{a_2^{qQ}c_3}{c_2^{qQ}a_3} = \frac{c_2a_3^{qQ}}{a_2c_3^{qQ}}.\]
	Setting $a_3 = \omega_1 a_2$, $c_3 = \omega_2c_2$ gives
	\[ - \frac{a}{B} = \left(\frac{a_2}{c_2}\right)^{qQ-1}\frac{\omega_2}{\omega_1} = \left(\frac{a_2}{c_2}\right)^{qQ-1}	\left(\frac{\omega_1}{\omega_2}\right)^{qQ}.\]
	This implies $\omega_1^{qQ+1} = \omega_2^{qQ+1}$, that is $\omega_2 = \zeta \omega_1$ where $\zeta$ is a $(qQ+1)^\textrm{st}$ root of unity. Substituting this into the previous equation yields
	\begin{equation}
		- \frac{a}{B} = \left(\frac{a_2}{c_2}\right)^{qQ-1}\zeta.
	\label{eq:f4_casenonzero}
	\end{equation}
	By Lemmas~\ref{lem:gcd} and \ref{lem_thm} (iii), we have $\Gcd{qQ+1}{p^m-1}=p^{\Gcd{k}{m}/2}+1$, so $\zeta$ is a $(p^{\Gcd{k}{m}/2}+1)^\textrm{st}$ root of unity. In particular, $\zeta \in \E$ since $p^{\Gcd{k}{m}/2}+1$ divides $p^{\Gcd{k}{m}}-1$. The $(p^{\Gcd{k}{m}/2}+1)^\textrm{st}$ roots of unity in $\E$ are precisely the $(p^{\Gcd{k}{m}/2}-1)^\textrm{st}$ powers in $\E$. In particular, $\zeta$ is a square. This is however a contradiction to Eq.~\eqref{eq:f4_casenonzero} since the left hand side is a non-square (recall that $-1$ and $a$ are squares), and the right hand side is a square. We conclude that $a_2c_2a_3c_3=0$. We can proceed identically with $b_2,d_2,b_3,d_3$ and Eqs.~\eqref{eq:F4_r32} and \eqref{eq:F4_r42} which yields $b_2d_2b_3d_3=0$. The conditions in Eqs.~\eqref{eq:F4_l12}, \eqref{eq:F4_l22}, \eqref{eq:F4_r32}, \eqref{eq:F4_r42} and the bijectivity of $L,M$ then only leave two possibilities: Either $a_2=a_3=d_2=d_3=0$ and $b_2,b_3,c_2,c_3 \neq 0$ or, the other way round, $a_2,a_3,d_2,d_3 \neq 0$ and $b_2=b_3=c_2=c_3=0$. We will deal with these two cases separately. Note that in both cases, Eqs. \eqref{eq:F4_l3}, \eqref{eq:F4_l4}, \eqref{eq:F4_r1}, \eqref{eq:F4_r2}, \eqref{eq:F4_l12}, \eqref{eq:F4_l22}, \eqref{eq:F4_r32}, \eqref{eq:F4_r42} are always satisfied. \\
	
	\textbf{Case $b_2,b_3,c_2,c_3=0$:}
	We set $a_3 = \omega_1 a_2$, $d_3 = \omega_2 d_2$. Then Eqs.~\eqref{eq:F4_l1}, \eqref{eq:F4_l2}, \eqref{eq:F4_r3}, \eqref{eq:F4_r4} become
	\[a_1 = a_2^{q+1}\omega_1 = a_2^{q+1} \omega_1^q = d_2^{q+1} \omega_2 = d_2^{q+1} \omega_2^q, \]
	which is satisfied if and only if $\omega_1,\omega_2 \in \E$ and 
	\begin{equation}
		(a_2/d_2)^{q+1} = \omega_2/\omega_1. 
	\label{eq:F4_cent_case1}
	\end{equation}

	Similarly, from Eq.~\eqref{eq:F4_l32}, \eqref{eq:F4_l42}, \eqref{eq:F4_r12}, \eqref{eq:F4_r22}, we get immediately (using $\omega_1,\omega_2 \in \E$),
\[
d_1 = a_2^{qQ}d_2 \omega_2 = a_2d_2^{qQ} \omega_2^{Q} = a_2d_2^{qQ} \omega_1 = a_2^{qQ}d_2 \omega_1^{Q}.\]
	This is equivalent to $\omega_1 = \omega_2^Q$, $(a_2/d_2)^{qQ-1} = \omega_2^{Q-1}$. Multiplying this with Eq.~\eqref{eq:F4_cent_case1} gives
		\[(a_2/d_2)^{q(Q+1)} = 1, \]
	i.e., $a_2/d_2 \in (\M^\times)^{Q-1}$, say $\zeta^{Q-1} = a_2/d_2$. Rewriting Eq.~\eqref{eq:F4_cent_case1} gives
	\[(a_2/d_2)^{q+1} = (\zeta^{q+1})^{Q-1} = \omega_1^{Q-1}.\]
	The equation cannot be satisfied if $\omega_1$ is a non-square. If $\omega_1$ is a square, then $d_2$ is uniquely determined up to the sign from $\omega_1$ and $a_2$. Thus we have in total $p^m-1$ choices for $a_2$, $(p^{\Gcd{k}{m}}-1)/2$ choices for $\omega_1$ and $2$ choices for $d_2$, making in total $(p^m-1)(p^{\Gcd{k}{m}}-1)$ choices in this case.\\
	\textbf{Case $a_2,a_3,d_2,d_3=0$:}	Similarly to the previous case, we set $b_3 =\omega_1 b_2$, $c_3 = \omega_2 c_2$. We get from the first set of equations:
		\[a_1 = Bc_2^{q+1}\omega_2 = Bc_2^{q+1} \omega_2^q = (1/B)b_2^{q+1} \omega_1 = (1/B)b_2^{q+1} \omega_1^q, \]
	which is equivalent to $\omega_1,\omega_2 \in \E$ and 
	\begin{equation}
		(b_2/c_2)^{q+1} = B^2\omega_2/\omega_1. 
	\label{eq:F4_cent_case2}
	\end{equation}
	The second set of equations gives
		\[d_1 = (a/B)c_2^{qQ}b_2 \omega_1 = (B/a)c_2b_2^{qQ} \omega_1^{Q} = (B/a)c_2b_2^{qQ} \omega_2 = (a/B)c_2^{qQ}b_2 \omega_2^{Q}.  \]
	This again implies $\omega_1^Q = \omega_2$ and $(b_2/c_2)^{qQ-1} = (1/\omega_1^{Q-1})(a/B)^2$. Multiplying this with Eq.~\eqref{eq:F4_cent_case2} gives
	\[(b_2/c_2)^{q(Q+1)} = \frac{\omega_2}{\omega_1^Q} a^2 = a^2 = a^{Q+1}. \]
	Thus, $b_2/c_2$ is determined up to multiplication with a $(Q+1)^\textrm{st}$ root of unity, that is, a $(Q-1)^\textrm{st}$ power, say $(b_2/c_2)^q = \zeta^{Q-1} a$ . Eq.~\eqref{eq:F4_cent_case2} can be rewritten as
	\[(b_2/c_2)^{q+1} = a^{\overline{q}+1}(\zeta^{\overline{q}+1})^{Q-1}  = B^2\omega_1^{Q-1}. \]
	This equation has a solution if and only if $(a^{\overline{q}+1})/B^2$ is a $(Q-1)^\textrm{st}$ power, say $(a^{\overline{q}+1})/B^2=\rho^{Q-1}$. In this case, there are $(p^{\Gcd{k}{m}}-1)/2$ possible choices for $\omega_1$, either all squares or all non-squares, depending on if $\rho$ is a square or not. Then $\zeta$ is determined up to the sign, so there are $p^m-1$ possible choices for $b_2$, $2$ possible choices for $c_2$ and $(p^{\Gcd{k}{m}}-1)/2$ possible choices for $\omega_1$. This case thus contributes either $0$ or $(p^m-1)(p^{\Gcd{k}{m}}-1)$ elements. Both cases together show that $|C_{\P}|$ is either $(p^m-1)(p^{\Gcd{k}{m}}-1)$ or $2(p^m-1)(p^{\Gcd{k}{m}}-1)$. 
	
Now it is clear that $p' \nmid [C_{\P} : Z^{(q,r)}] \in \{p^{\Gcd{k}{m}}-1,2(p^{\Gcd{k}{m}}-1)\}$
by our assumption that $p'$ is $p$-primitive (recall $p' \ne 2$).
\end{proof}

We can now apply Thereom~\ref{thm:equivalence} to the pre-semifields in the Family $\mathcal{S}$. 

\begin{theorem} \label{thm:f1_equiv}
Let $\P_{q,B,a}=(\M \times \M,+,\ast_1)$ and $\P_{q',B',a'}=(\M \times \M,+,\ast_2)$ be pre-semifields from the Family $\mathcal{S}$. Then 
\begin{enumerate}[(i)]

	\item $\P_{q,B,a}$  and  $\P_{q',B',a'}$ are isotopic if and only if they are strongly isotopic.
	\item $\P_{q,B,a}$ is isotopic to $\P_{\overline{q},B,a'}$ for $a'=B^{Q+1}/a$ and arbitrary $q$.
	\item $\P_{q,B,a}$ is isotopic to $\P_{q,B',a'}$ for arbitrary $q,B,B',a$ and a suitable choice for $a'$. 
	\item If $\P_{q,B,a}$ is isotopic to $\P_{q,B,a'}$, then it is also isotopic to $\P_{q,B,-a'}$.
	\item There are at most $2m=n$ different $a'$ such that $\P_{q,B,a}$ is isotopic to $\P_{q,B,a'}$.
	\item No other isotopisms exist.
\end{enumerate}
\end{theorem}
\begin{proof}
Let $(N,L,M)$ be an isotopism between $\P_{q,B,a}$ and $\P_{q',B',a'}=(\M \times \M,+,\ast_2)$. All subfunctions of $N,L,M$ are zero or monomials by Theorem~\ref{thm:equivalence}. Moreover, $\P_{q,B,a}$ and $\P_{q',B',a'}$ can only be isotopic if $q'=q$, $q'=\overline{q}$, $q' = qQ$, or $q'=\overline{q}Q$. Note that if $m/\Gcd{k}{m}$ is odd, then $m/\Gcd{k+m/2}{m}$ is even by Lemma~\ref{lem_thm} (iii), so the cases $q' = qQ$, $q'=\overline{q}Q$ do not satisfy the conditions of Theorem~\ref{thm_family1} and need not be considered.\\

We first show the isotopy in the case $q'=\overline{q}$. We have
	\[(x,y)\ast_1(u,v) = (x^qu+xu^q+B(y^qv+yv^q),x^{qQ}v+yu^{qQ}+(a/B)(xv^{qQ}+y^{qQ}u)).\]
	A transformation with 
	\[
	N_1=x, N_4=(B^Q/a)x^Q, N_2=N_3=0
	\]
	and raising $x,y,u,v$ to the $\overline{q}$-th power yields
	\[
	N((x^{\overline{q}},y^{\overline{q}})\ast_1(u^{\overline{q}},v^{\overline{q}})) = (xu^{\overline{q}}+x^{\overline{q}}u+B(yv^{\overline{q}}+y^{\overline{q}}v),(B^Q/a)(xv^{\overline{q}Q}+y^{\overline{q}Q}u)+x^{\overline{q}Q}v+yu^{\overline{q}Q}). \]
	Observe that one can write $B^Q/a = a'/B$ for some $a'\in \L$ (indeed this is equivalent to $B^{Q+1} = aa'$ which has always a solution since $B^{Q+1} \in \L$). We conclude that there is always a strong isotopism between $\P_{q,B,a}$ and $\P_{\overline{q},B,B^{Q+1}/a}$. 
Thus we have proved Part (ii) of the theorem.\\

It thus only remains to deal with the case $q'=q$. By Theorem~\ref{thm:equivalence}, we only need to consider isotopisms $(N,L,M)$ with subfunctions 
\begin{align*}
& N_1 = a_1x^{p^{t}}, \quad N_4 = d_1x^{p^{t}},\\ 
& L_{1}=a_2x^{p^{t}},\quad L_{2}=b_2x^{p^{t}},\quad L_{3}=c_2x^{p^{t}},\quad L_{4}=d_2x^{p^{t}}, \\
& M_{1}=a_3x^{p^{t}},\quad M_{2}=b_3x^{p^{t}},\quad M_{3}=c_3x^{p^{t}},\quad M_{4}=d_3x^{p^{t}},
\end{align*}
for some $t \in \{0,\dots,m-1\}$. Then 
	\begin{align*}
		L(x,y)\ast_2 M(u,v) =& (((a_2'x+b_2'y)^{q}(a_3'u+b_3'v)+(a_2'x+b_2'y)(a_3'u+b_3'v)^{q})^{p^t}\\&+B'((c_2'x+d_2'y)^{q}(c_3'u+d_3'v)+(c_2'x+d_2'y)(c_3'u+d_3'v)^{q})^{p^t},\\&((a_2'x+b_2'y)^{qQ}(c_3'u+d_3'v)+(c_2'x+d_2'y)(a_3'u+b_3'v)^{qQ})^{p^t}\\&+(a'/B')((a_2'x+b_2'y)(c_3'u+d_3'v)^{qQ}+(c_2'x+d_2'y)^{qQ}(a_3'u+b_3'v))^{p^t}),
	\end{align*}
	where $a_i'=a_i^{p^{m-t}}$ and similarly for the other coefficients $b_i,c_i,d_i$.
	We also obtain 
		\[N((x,y)\ast_1(u,v)) = (a_1(x^qu+xu^q+B(y^qv+yv^q))^{p^{t}},d_1(x^{qQ}v+yu^{qQ}+(a/B)(xv^{qQ}+y^{qQ}u))^{p^{t}}).\]
		We compare the coefficients $(x^{q}u)^{p^{t}}$, $(xu^{q})^{p^{t}}$, $(x^{q}v)^{p^{t}}$, $(xv^{q})^{p^{t}}$, $(y^{q}u)^{p^{t}}$, $(yu^{q})^{p^{t}}$, $(y^{q}v)^{p^{t}}$, $(yv^{q})^{p^{t}}$ in the first component to get the following $8$ equations.	
		
 \noindent\begin{minipage}{0.5\textwidth}
\begin{align}
	a_1 &= a_2^qa_3+B'c_2^qc_3 \label{eq:F42_l1}\\
	a_1&=a_2a_3^q+B'c_2c_3^q \label{eq:F42_l2}\\	
	0 &= a_2^qb_3+B'c_2^qd_3\label{eq:F42_l3}\\
	0 &= a_2b_3^q+B'c_2d_3^q\label{eq:F42_l4} 
\end{align}
    \end{minipage}%
    \begin{minipage}{0.5\textwidth}
	\begin{align}
	0 &= b_2^qa_3+B'd_2^qc_3\label{eq:F42_r1} \\
	0 &= b_2a_3^q+B'd_2c_3^q\label{eq:F42_r2}\\
	B^{p^t}a_1 &= b_2^qb_3+B'd_2^qd_3 \label{eq:F42_r3}\\
	B^{p^t}a_1 &= b_2b_3^q+B'd_2d_3^q.\label{eq:F42_r4}
	\end{align}
    \end{minipage}\vskip1em
	And similarly the $8$ equations that come from comparing the coefficients in the second component:
	
 \noindent\begin{minipage}{0.5\textwidth}
\begin{align}
	0 &= a_2^{qQ}c_3+(a'/B')c_2^{qQ}a_3 \label{eq:F42_l12}\\
	0&=(a'/B')a_2c_3^{qQ}+c_2a_3^{qQ}\label{eq:F42_l22}\\	
	d_1 &= a_2^{qQ}d_3+(a'/B')c_2^{qQ}b_3\label{eq:F42_l32}\\
	d_1(a/B)^{p^t} &= (a'/B')a_2d_3^{qQ}+c_2b_3^{qQ}\label{eq:F42_l42} 
\end{align}
    \end{minipage}%
    \begin{minipage}{0.5\textwidth}
	\begin{align}
	d_1(a/B)^{p^t} &= b_2^{qQ}c_3+(a'/B')d_2^{qQ}a_3\label{eq:F42_r12} \\
	d_1 &= (a'/B')b_2c_3^{qQ}+d_2a_3^{qQ}\label{eq:F42_r22}\\
	0 &= b_2^{qQ}d_3+(a'/B')d_2^{qQ}b_3 \label{eq:F42_r32}\\
	0 &= (a'/B') b_2d_3^{qQ}+d_2b_3^{qQ}.\label{eq:F42_r42}
	\end{align}
    \end{minipage}\vskip1em
	Note that Eqs.~\eqref{eq:F42_l12}, \eqref{eq:F42_l22}, \eqref{eq:F42_r32}, \eqref{eq:F42_r42} are identical to Eqs.~\eqref{eq:F4_l12}, \eqref{eq:F4_l22}, \eqref{eq:F4_r32}, \eqref{eq:F4_r42} in the proof of Lemma~\ref{lem:centralizer_F1}, just with $a/B$ substituted by $(a'/B')$. We can thus conclude with the same reasoning as in the proof of Lemma~\ref{lem:centralizer_F1} that either $b_2=b_3=c_2=c_3=0$ or $a_2=a_3=d_2=d_3=0$. \\
	\textbf{Case $b_2=b_3=c_2=c_3=0$:} Here, we also proceed similarly to the proof of Lemma~\ref{lem:centralizer_F1}. We set $a_3 = \omega_1 a_2$, $d_3 = \omega_2 d_2$ and get from Eqs.~\eqref{eq:F42_l1}, \eqref{eq:F42_l2}, \eqref{eq:F42_r3}, \eqref{eq:F42_r4}

\[a_1 = a_2^{q+1}\omega_1 = a_2^{q+1} \omega_1^q = (B'/B^{p^t})d_2^{q+1} \omega_2 = (B'/B^{p^t})d_2^{q+1} \omega_2^q, \]
	which is satisfied if and only if $\omega_1,\omega_2 \in \E$ and 
	\begin{equation}
		\left(\frac{a_2}{d_2}\right)^{q+1} = \frac{\omega_2B'}{\omega_1B^{p^t}}. 
	\label{eq:F42_cent_case1}
	\end{equation}

	Similarly, from Eq.~\eqref{eq:F42_l32}, \eqref{eq:F42_l42}, \eqref{eq:F42_r12}, \eqref{eq:F42_r22}, we get immediately (using $\omega_1,\omega_2 \in \E$)
	\[d_1 = a_2^{qQ}d_2 \omega_2 = \frac{a'B^{p^t}}{a^{p^t}B'}a_2d_2^{qQ} \omega_2^{Q} = \frac{a'B^{p^t}}{a^{p^t}B'}a_2d_2^{qQ} \omega_1 = a_2^{qQ}d_2 \omega_1^Q.  \]
	This is equivalent to $\omega_1^Q = \omega_2$ and $(a_2/d_2)^{qQ-1} = \omega_2^{Q-1}(a'B^{p^t})/(a^{p^t}B')$. 
	Multiplying the second condition with Eq.~\eqref{eq:F42_cent_case1} gives
			\begin{equation}
		\left(\frac{a_2}{d_2}\right)^{q(Q+1)} = \frac{\omega_2^Qa'}{\omega_1a^{p^t}} = \frac{a'}{a^{p^t}}. 
				\label{eq:F1_almostlastequation}
		\end{equation}

		Using $\omega_1^Q = \omega_2$, we rewrite Eq.~\eqref{eq:F42_cent_case1}:
		\begin{equation}
					\left(\frac{a_2}{d_2}\right)^{q+1} = \frac{\omega_1^{Q-1}B'}{B^{p^t}}. 
		\label{eq:F1_lastequation}
		\end{equation}

	Observe that $B,B',t,\omega_1$ uniquely determine $(a_2/d_2)$ up to the sign from Eq.~\eqref{eq:F1_lastequation}. Since $(a_2/d_2)^{q(Q+1)} \in \L$, there is thus for each $B,B',t,\omega_1,a$ precisely one $a'$ that satisfies all conditions. For all $\omega_1$ that are squares (i.e., all $(q+1)^\textrm{st}$ powers), this $a'$ is the same since $\omega_1^{(Q-1)(Q+1)}=1$. Similarly, for all $\omega_1$ that are non-squares, we have $\omega_1^{(Q-1)(Q+1)/2}=-1$, so they also all yield the same $a'$, and in fact precisely the same $a'$ as when $\omega_1$ is a square, just with different sign. 
In particular, we conclude that a pre-semifield $\P_{q,B,a}$ is always isotopic 
to $\P_{q,B',a'}$ for arbitrary $B'$ and a suitable choice of $a'$. Since we can choose $\omega_1 = \omega_2 = 1$, we can even choose $a'$ such that the pre-semifields are strongly isotopic.
Thus, we have proved Part (iii) of our theorem.\\

 Consequently, it is enough to consider 
isotopisms in the case $B=B'$ for an arbitrary non-square $B$. 
When $B=B'$, every possible choice of $t$ yields an $a'$ such that a pre-semifield 
$\P_{q,B,a}$ is strongly isotopic to $\P_{q,B,a'}$ and isotopic to $\P_{q,B,-a'}$. 
Assume the choice of $t$ in the previously described procedure leads to a 
strong isotopy between $\P_{q,B,a}$ and $\P_{q,B,a'}$. We now show that choosing 
$t'$ defined by $t'-t \equiv m/2 \pmod m$ in the same procedure gives a strong 
isotopy to $\P_{q,B,-a'}$, i.e. $\P_{q,B,a}$ and $\P_{q,B,-a'}$ are not just 
isotopic but also strongly isotopic. 
	
	Let $(a_2/d_2)^{q+1}$ be determined by $\omega_1=1$ and fixed $B=B'$, $t$ via Eq.~\eqref{eq:F1_lastequation}, i.e.
	\[ \left(\frac{a_2}{d_2}\right)^{q+1} = \frac{1}{B^{p^t-1}}.\]
	Similarly, let $(a_2'/d_2')^{q+1}$ be determined by $\omega_1=1$, the same $B=B'$ and $t'$:
		\[ \left(\frac{a_2'}{d_2'}\right)^{q+1} = \frac{1}{B^{p^{t'}-1}}.\]
	We then have 
	\[\left(\frac{a_2'}{d_2'}\right)^{q+1} = \left(\frac{a_2}{d_2}\right)^{q+1} \cdot \frac{1}{B^{p^{t'}-p^t}}.\]
	Since $B^{p^{t'}-p^t}=(B^{Q-1})^{p^t} \in (\M^\times)^{Q-1}$, we have $(a_2'/d_2')^{q+1} = (\zeta a_2/d_2)^{q+1}$ where $\zeta^{q+1}=1/(B^{p^{t'}-p^t}) \in (\M^\times)^{Q-1}$. Note that $\zeta \notin (\M^\times)^{Q-1}$ since $B$ is a non-square, so $B^{(Q-1)(Q+1)/2} \neq 1$. In particular, we have $\zeta^{Q+1}=-1$. Then by Eq.~\eqref{eq:F1_almostlastequation}, $(a_2'/d_2')^{q(Q+1)} =\zeta^{q(Q+1)} (a_2/d_2)^{q(Q+1)} =-(a_2/d_2)^{q(Q+1)}=-a'/a^{p^t}$. We conclude that $\P_{q,B,a}$ and $\P_{q,B,a'}$ are strongly isotopic if and only if $\P_{q,B,a}$ and $\P_{q,B,-a'}$ are strongly isotopic. \\
	\textbf{Case $a_2,a_3,d_2,d_3=0$:}	Similarly to the previous case, we set $b_3 =\omega_1 b_2$, $c_3 = \omega_2 c_2$. Since we know from the previous case that different $B,B'$ always lead to strongly isotopic pre-semifields (for suitable choices of $a,a'$), we only consider the case $B=B'$ without loss of generality. We get from the first set of equations:
		\[a_1 = Bc_2^{q+1}\omega_2 = Bc_2^{q+1} \omega_2^q = (1/B^{p^t})b_2^{q+1} \omega_1 = (1/B^{p^t})b_2^{q+1} \omega_1^q, \]
	which is equivalent to $\omega_1,\omega_2 \in \E$ and 
	\begin{equation}
		\left(\frac{b_2}{c_2}\right)^{q+1} = B^{p^t+1}\cdot \frac{\omega_2}{\omega_1}.
	\label{eq:F42_cent_case2}
	\end{equation}
	The second set of equations gives
		\[d_1 = (a'/B)c_2^{qQ}b_2 \omega_1 = (B/a)^{p^t}c_2b_2^{qQ} \omega_1^{Q} = (B/a)^{p^t}c_2b_2^{qQ} \omega_2 = (a'/B)c_2^{qQ}b_2 \omega_2^{Q}.  \]
	This again implies $\omega_1^Q = \omega_2$ and $(b_2/c_2)^{qQ-1} = (1/\omega_1^{Q-1})(a/B)^{p^t}(a'/B)$. Multiplying this with Eq.~\eqref{eq:F42_cent_case2} gives
	\[\left(\frac{b_2}{c_2}\right)^{q(Q+1)} = \frac{\omega_2}{\omega_1^Q} a^{p^t}a' = a^{p^t}a'. \]
	Eq.~\eqref{eq:F42_cent_case2} can be rewritten as
	\[\left(\frac{b_2}{c_2}\right)^{q+1} = \omega_1^{Q-1}B^{p^t+1}. \]
	These two equations are structurally identical to Eqs.~\eqref{eq:F1_almostlastequation} and \eqref{eq:F1_lastequation} from the previous case. With the same argumentation, we conclude that every possible choice of $t$ yields an $a'$ such that a pre-semifield $\P_{q,B,a}$ is strongly isotopic to $\P_{q,B,a'}$ and isotopic to $\P_{q,B,-a'}$. Again, choosing $t'$ such that $t'-t \equiv m/2 \pmod m$ gives also strong isotopy between $\P_{q,B,a}$ and $\P_{q,B,-a'}$. 
This proves Parts (iv) and (i) of our theorem.
Now we can simply prove Parts (v) and (vi). Considering both cases together, there are thus at most $2m=n$ different $a'$ such that $\P_{q,B,a}$ is strongly isotopic to $\P_{q,B,a'}$. 
We have considered all cases thus there are no more isotopisms. 
\end{proof}

\begin{remark}[Planar equivalence and strong isotopy]

Instead of the exposition we chose based on isotopy, we could have developed an approach 
based on (in)equivalences of planar $(q,r)$-biprojective mappings.
Recall that Theorem~\ref{thm:eaequiv} states that strong isotopy of pre-semifields 
corresponds to equivalence of the corresponding planar mappings. 
We give a very brief sketch of such an approach: One can define the 
automorphism group $\Aut(F)$ of a planar mapping $F$ 
of $\F \cong \M \times \M$ from a DO polynomial 
as the set of all $(N,L) \in \GL(\F)^2$ such that $N F  L^{-1} = F$. 
Note that, by Theorem~\ref{thm:eaequiv}, $\Aut(F) \cong \Aut_S(\P)$, where 
$\Aut_S(\P)$ is the group of all strong autotopisms of the pre-semifield $\P$ 
belonging to $F$. It is then clear (identically to Lemma~\ref{lem_first}) that the set 
\[
Z_F^{(q,r)} = \{(\diag(m_{a^{q+1}},m_{a^{r+1}}), 
                \diag(m_a,m_a)) \colon a \in \M^\times\} \cong Z^{(q,r)}
\]
is a subgroup of $\Aut(F)$. The same group theoretic machinery can then be applied, 
with $Z_F^{(q,r)}$ and $\Aut(F)$ taking the role of $Z^{(q,r)}$ and $\Aut(\P)$ in 
the approach we presented,  
proving an analogue of Theorem~\ref{thm:equivalence}, 
with the conclusion that all subfunctions of $N$ and $L$
are zeros or monomials. 

Then, the planar mappings from Family $\Family 1$ can be tested for equivalence similar to 
Theorem~\ref{thm:f1_equiv} by comparing the coefficients of the polynomial equation $N F = F L$. One obtains the same set of equations as Eqs.~\eqref{eq:F42_l1}-\eqref{eq:F42_r42}, just with the simplification that $M = L$. Then the same argumentation of the proof of Theorem~\ref{thm:f1_equiv} can be followed, with the result that  the only possible equivalences that need to be considered are equivalences via $N,L$ where the subfunctions are
\[N_1 = a_1 x^{p^t}, \quad N_2 = N_3 = 0, \quad N_4 = d_1 x^{p^t}\]
and either
\[L_1 = a_2 x^{p^t}, \quad L_2 = L_3 = 0, \quad L_4 = d_2 x^{p^t}\]
or 
\[L_1 = 0, \quad L_2 = b_2 x^{p^t}, \quad L_3 = c_2 x^{p^t}, \quad L_4 = 0.\]
The conditions on the coefficients are then identical to the ones in the proof of Theorem~\ref{thm:f1_equiv} (e.g. Eq.~\eqref{eq:F1_almostlastequation} and~\eqref{eq:F1_lastequation}), just with $\omega_1=\omega_2=1$. This way, one obtains the same result as Theorem~\ref{thm:f1_equiv}, except that one only gets information on \emph{strong} isotopy and not regular isotopy.

Since an isotopy class of a commutative semifields contains at most $2$ strong-isotopy
classes~\cite[Theorem 2.6.]{coulterhenderson}, this approach would suffice to prove
the exponential count. With some more effort, the planar mapping approach can also
be used to find all isotopisms between commutative semifields (not just strong isotopisms). 
Indeed, by~\cite[Theorem 2.6.]{coulterhenderson}, if two commutative semifields with 
corresponding planar mapping $F : x \mapsto x \ast_1 x$ and $G : x \mapsto x\ast_2 x$ 
are isotopic, then either $F$ and $G$ are equivalent (the semifields are then strongly 
isotopic) or $F$ is equivalent to $G' : x \mapsto x \ast_2(a \ast_2 x)$ where 
$a$ is an arbitrary non-square element 
in the middle nucleus of the semifield. So in order to settle the isotopy question one 
could check equivalence between $F$ and both $G$ and $G'$, yielding an alternative proof 
of Theorem~\ref{thm:equivalence}. The two approaches are essentially equivalent and 
require similar effort.


The isotopy approach we chose has the advantage that it can be extended naturally to 
non-commutative semifields where the connection to planar mappings does not exist. 

\end{remark}

The number of distinct isotopy classes can now be counted.

\begin{corollary} \label{cor:f1_equiv2}
 Let $N_{\mathcal{S}}(p,n)$ be the number of non-isotopic pre-semifields in 
 Family $\mathcal{S}$ on $\F_p^n$. Then
\[
\frac{\sigma(n)-1}{2}\cdot \frac{p^{n/4}-1}{n}\leq N_{\mathcal{S}}(p,n) \leq \frac{\sigma(n)-1}{2}\left(p^{n/4}-1\right).\]
\end{corollary}
\begin{proof}
This follows directly from Theorem~\ref{thm:f1_equiv} (i),(ii),(iv) and (v): There are $\sigma(n)-1$ admissible values for $q$, and only $q,\qbar$ yield isotopic pre-semifields. Then there are $p^{n/4}-1$ admissible values for $a$, with at most $n$ of them yielding isotopic pre-semifields.
\end{proof}

In particular, $\mathcal{S}$ is the first known family of commutative 
(pre-)semifields that yields exponentially many non-isotopic (pre-)semifields. 
Since non-isotopic pre-semifields lead to inequivalent planar mappings 
(see Theorem~\ref{thm:eaequiv}), this also shows that the number of inequivalent 
planar mappings grows exponentially in $n$.
\begin{corollary}\label{cor:exponential}
The number of non-isotopic commutative semifields of order $p^n$ and
the number of inequivalent planar DO mappings of $\f{p^n}$ are exponential in $n$ 
for a fixed odd prime $p$ and $n$ divisible by $4$. 
\end{corollary}


\section{The nuclei} \label{sec_nuclei}

In this section we will compute the nuclear parameters of Family $\Family 1$. As explained
in Section \ref{sec_prem}, the nuclei are defined for semifields and not for pre-semifields. However,
the nuclei of the isotopic semifield can be computed using the following theorem of 
Marino and Polverino \cite[Theorem 2.2]{MP12} (we give the commutative version of their general 
theorem) that allows computing the nuclei directly from the pre-semifield.  

Let $\P = (\f{p}^n,+,\ast)$ be a commutative pre-semifield with right 
multiplication defined as
\[
	R_U : X \mapsto X \ast U, \quad \textrm{for } U \in \f{p}^n.
\]
Then the \textbf{spread set} associated to $\P$ is defined as
\[
	\ML = \{ R_U : U \in \f{p}^n \}.
\]
In the following 
$\N_j(\P)$ denotes the corresponding nucleus of the semifield isotopic to $\P$,
for $j \in \{l,m,r\}$.

\begin{theorem}\cite[Theorem 2.2]{MP12}\label{MP}
Let $\MN_0,\MN_1 \subset \End(\f{p}^n)$ be the largest sets (and then necessarily fields) such that
\[
	\ML \MN_0 \subseteq \ML \textrm{ and }  \MN_1 \ML \subseteq \ML.
\]
Then
\[
	\N_m(\P) \cong \MN_0 \textrm{ and } \N_l(\P) = \N_r(\P) \cong \MN_1.
\]
\end{theorem}

Now, let $\P = (\M \times \M,+,\ast)$ be a pre-semifield in Family $\Family 1$. 
Then $\ML = \{ R_{u,v} : (u,v) \in \M \times \M \}$, where
\[
	R_{u,v} : (x,y) \mapsto (R^{(1)}_{u,v}(x,y),R^{(2)}_{u,v}(x,y)),
\]
with
\[
	R^{(1)}_{u,v}(x,y) = x^qu+xu^q+B(y^qv+yv^q) \textrm{ and } R^{(2)}_{u,v}(x,y) =  x^rv+Axv^r+Ay^ru+yu^r.
\]
We write again $L \in \End(\M \times \M)$  as 
$L : (x,y) \mapsto (\aa(x)+\bb(y),\cc(x)+\dd(y))$,
where $\aa,\bb,\cc,\dd \in \End(\M)$.

\begin{theorem}\label{thm_nucleus}
The left, middle and right nuclei $\N_l(\P),\N_m(\P),\N_r(\P)$ satisfy $\N_l(\P) = \N_r(\P) \cong \D$ and $\N_m(\P) \cong \E$.
\end{theorem}
\begin{proof}
We have nonzero $L \in \MN_1$, if and only if, 
for every $(u,v) \in \M \times \M$ there exists $(w,t) \in \M \times \M$ such that
\begin{align*}
\aa R^{(1)}_{u,v} + \bb R^{(2)}_{u,v} &= R^{(1)}_{w,t}, \textrm{ and},\\
\cc R^{(1)}_{u,v} + \dd R^{(2)}_{u,v} &= R^{(2)}_{w,t},
\end{align*}
that is
\begin{align*}
\aa(x^qu+xu^q+B(y^qv+yv^q)) + \bb(x^rv+Axv^r+Ay^ru+yu^r) &= x^qw+xw^q+B(y^qt+yt^q), \textrm{ and},\\
\cc(x^qu+xu^q+B(y^qv+yv^q)) + \dd(x^rv+Axv^r+Ay^ru+yu^r) &= x^rt+Axt^r+Ay^rw+yw^r.
\end{align*}
This implies (after a routine comparison of degrees of $x,y,u,v$ as in previous sections) that $ \bb = \cc = 0$ and
$\aa(x) = z_1 x$ and $\dd(x) = z_4 x$ for some $z_1,z_4 \in \M^\times$. Now, the above equations become
\begin{align*}
z_1(x^qu+xu^q+B(y^qv+yv^q)) &= x^qw+xw^q+B(y^qt+yt^q), \textrm{ and},\\
z_4(x^rv+Axv^r+Ay^ru+yu^r)  &= x^rt+Axt^r+Ay^rw+yw^r,
\end{align*}
or (the case $uv = 0$ is easy to see)
\begin{align*}
z_1 &= w/u = (w/u)^q = t/v = (t/v)^q, \textrm{ and},\\
z_4 &= w/u = (w/u)^r = t/v = (t/v)^r.
\end{align*}
Thus for every $(u,v) \in \M \times \M$ there exists $(w,t) \in \M \times \M$ if and only if 
$z_4 = z_1 = z_1^q = z_1^r$ if and only if $z_1 \in \f{q} \cap \f{r} \cap \M = \D$. 
That is to say $L \in \MN_1$ if and only if $L(x,y)=(zx,zy)$ for $z \in \D$.
Now Theorem \ref{MP} implies $\N_l(\P) = \N_r(\P) \cong \D$.

Similarly for the middle nucleus, nonzero $L \in \MN_0$, if and only if,
\begin{align*}
R^{(1)}_{u,v}(\aa(x)+\bb(y),\cc(x)+\dd(y)) &= R^{(1)}_{w,t}(x,y), \textrm{ and},\\
R^{(2)}_{u,v}(\aa(x)+\bb(y),\cc(x)+\dd(y)) &= R^{(2)}_{w,t}(x,y),
\end{align*}
that is
\begin{align*}
(\aa(x)+\bb(y))^qu+(\aa(x)+\bb(y)) u^q + B(\cc(x)+\dd(y))^qv + B(\cc(x)+\dd(y)) v^q &= x^qw+xw^q+B(y^qt+yt^q), \textrm{ and},\\
(\aa(x)+\bb(y))^rv + A(\aa(x)+\bb(y)) v^r +A (\cc(x)+\dd(y))^ru + (\cc(x)+\dd(y)) u^r &= x^rt+Axt^r+Ay^rw+yw^r.
\end{align*}
This implies (after a routine comparison of degrees of $x,y,u,v$) that 
$\aa(x) = z_1 x, \bb(y) = z_2 y, \cc(x) = z_3 x$ and $\dd(y) = z_4 y$ for $z_1,z_2,z_3,z_4 \in \M$.
Now, the $x$-part of the first of the above equation implies
\[
z_1^qx^qu + z_1xu^q + B(z_3^qx^qv+z_3xv^q) = x^qw + xw^q,
\]
in other words,
\[
z_1^qu+Bz_3^qv = w \textrm{ and } z_1u^q + Bz_3v^q = w^q,
\]
which implies 
\[
(z_1^{q^2} - z_1) u^q + (B^q z_3^{q^2} - Bz_3) v^q = 0,
\]
for all $u,v \in \M$. That is to say $z_1 \in \f{q^2} \cap \M = \E$. The $x$-part 
of the second equation yields (after simple calculations)
\[
A^r z_3^{r^2} - \frac{z_3}{A} = 0.
\]
That is to say, if $z_3 \ne 0$ then,
\[
z_3^{q^2-1} = \frac{1}{B^{q-1}} \textrm{ and } z_3^{r^2-1} = \frac{1}{A^{r+1}}.
\]
By definition of $\Family 1$, $B$ is a non-square in $\M = \f{Q^2}$ and $A = a/B$
where $a \in \f{Q}^\times$. Recalling that $q^2 \equiv r^2 \pmod{Q^2}$, we reach
\[
B^{q(Q+1)} = B^{q+r} = a^{r+1}.
\]
Note that since $B$ is a non-square in $\f{Q^2}$ we have $B^{(Q+1)(Q-1)/2} = -1$ and
$B^{Q+1}$ is a non-square in $\f{Q}$. But $a^{r+1}$ is a square in $\f{Q}$ and we get
$z_3 = 0$. By the $y$-parts of the equations we similarly reach $z_2 = 0$ and $z_4 \in \E$.
Thus,
\begin{align*}
z_1^qx^qu+z_1xu^q+B(z_4^qy^qv+z_4yv^q)) &= x^qw+xw^q+B(y^qt+yt^q), \textrm{ and},\\
z_1^rx^rv+Az_1xv^r+Az_4^ry^ru+z_4yu^r)  &= x^rt+Axt^r+Ay^rw+yw^r,
\end{align*}
implying (the case $uv = 0$ is easy to see)
\begin{align*}
z_1^q = w/u &\textrm{ and } z_1 = (w/u)^q,\\
z_4^q = t/v &\textrm{ and } z_4 = (t/v)^q,\\
z_1^r = t/v &\textrm{ and } z_1 = (t/v)^r,\\
z_4^r = w/u &\textrm{ and } z_4 = (w/u)^r.
\end{align*}
Thus for every $(u,v) \in \M \times \M$ there exists $(w,t) \in \M \times \M$ if and only if 
$z_4^q = z_1 = z_1^{q^2} = z_1^{r^2}$  
if and only if $z_1 \in \f{q^2} \cap \M = \E$. That is to say $L \in \MN_0$ if and only if $L(x,y)=(zx,zy)$ for $z \in \E$.
Now Theorem \ref{MP} implies $\N_m(\P) \cong \E$.
\end{proof}


\section{Comparison to other commutative semifields and concluding remarks} \label{sec_equiv}

Table \ref{table_comm} lists known commutative semifields that are not 
biprojective. We should say here that these commutative semifields are not 
\textit{obviously} represented as biprojective semifields. When the order is square, 
there might be isotopic semifields that can be biprojective, but we are not aware of 
such isotopisms.

\begin{table}[!ht]
\noindent\begin{center} 
{\footnotesize
\begin{tabular}{|c|c|c|c|c|c|c|} 
\hline 
\textbf{Family} & \textbf{Planar Mapping} & $\#\S$ & \textbf{Notes} & 
$(\#\N_l,\#\N_m)$& \textbf{Count} & \textbf{Proved in}\\ 
\hline 
$\mathcal{ZKW}$             & $X^{q+1} - a^{Q-1}X^{qQ+Q^2}$  &$p^{3s}$  &  
\begin{tabular}{@{}c@{}}
$Q = p^s, \quad q = p^t,$\\
$ d = \Gcd{s}{t}, \quad s' = s/d, \quad t' = t/d,$\\
$s'$ odd, $\quad s' + t' \equiv 0 \pmod{3}$,\\
$\langle a \rangle = \f{p^{3s}}^\times$.
\end{tabular}
 &$(p^{d},p^{d})$  \cite{MP12}  & $\ge 1$ & \cite{ZKW}\\
\hline 
$\mathcal{B}_3$               &  $X^{q+1} - a^{Q-1}X^{qQ+Q^2}$  &$p^{3s}$  &  
\begin{tabular}{@{}c@{}}
$Q = p^s, q = p^t,$\\
$d = \Gcd{s}{t}, \quad s/d$ odd,\\ 
$q \equiv Q \equiv 1 \pmod{3}$,\\
$\langle a \rangle = \f{p^{3s}}^\times$.
\end{tabular}
 &$(p^{d},p^{d})$ \cite{MP12}  & $ \le 9\sigma(s)$ & \cite{Bierbrauer10}\\
\hline 
$\mathcal{B}_4$               &  $X^{q+1} - a^{Q-1}X^{qQ+Q^3}$  &$p^{4s}$  &  
\begin{tabular}{@{}c@{}}
$Q = p^s, q = p^t,$\\
$d = \Gcd{2s}{t}, \quad 2s/d$ odd,\\ 
$q \equiv Q \equiv 1 \pmod{4}$,\\
$\langle a \rangle = \f{p^{4s}}^\times$.
\end{tabular}
 &$(p^{d/2},p^{d})$ \cite{MP12} & $ \le 8\sigma(s)$ & \cite{Bierbrauer10}\\
\hline 
$\mathcal{G}$              &  
$(x^2 + y^{10},xy - y^6)$
&$3^{2m}$  & $m \ge 3$ odd & $(3,3^m)$ & $1$ & \cite{Ganley} \\
\hline 
$\mathcal{CG}$               & 
$(x^2 + ay^2 + a^3y^{18}, xy-ay^6)$
&$3^{2m}$ & $m \ge 3$, $a \in \f{3^m}^\times \setminus (\f{3^m}^\times)^2$ & $(3,3)$ & $1$ & \cite{CG} \\
\hline 
$\mathcal{CM}/\mathcal{DY}$ & 
$X^{10} \pm X^6 -X$ & $3^{m}$ & $m \ge 5$ odd & $(3,3)$ & $2$ & \cite{CM,DY}\\
\hline

\end{tabular} 
}
\end{center}
\caption{Known infinite families of (non-biprojective) commutative semifields of order $p^n$} \label{table_comm}
\end{table}

We now consider isotopisms between the new Family $\Family 1$ 
and other commutative pre-semifields.

\begin{theorem} \label{thm:inequiv_distinct}
	Let $\P_{q,B,a} =(\M \times \M,+,\ast)$ be a pre-semifield in the Family $\mathcal{S}$. 
	$\P_{q,B,a}$ is not isotopic to any other known commutative semifield, except possibly 
	semifields from Family $\mathcal{B}_4$. Family $\mathcal{S}$ yields new examples of 
	commutative semifields.
\end{theorem}
\begin{proof}
	The non-isotopy with the biprojective pre-semifields follows directly from 
	Theorem~\ref{thm:equivalence}, except for possible isotopisms between the 
	families $\mathcal{S}$ and $\mathcal{ZP}$ when the coefficients $q,r$ coincide. 
	We exclude this case by again applying Theorem~\ref{thm:equivalence}: Consider the Zhou-Pott
	pre-semifield $\P_\alpha = (\M \times \M,+,\star)$ with multiplication
\[
(x,y) \star (u,v) = (x^qu+u^qx + \alpha(y^qv+yv^{q}), x^{qQ}v+yu^{qQ})
\]
	for some (arbitrary) non-square $\alpha$. Note that it is not possible to use the 
	parameter $qQ$ in the first component and $q$ in the second component since 
	$\gcd(k+m/2,m)=\gcd(k,m)/2$ by Lemma~\ref{lem_thm} (iii), contradicting the necessary 
	conditions of a Zhou-Pott pre-semifield.
	If $\P_\alpha$ is isotopic to $\P_{q,B,a}$, then (using 
	Theorem~\ref{thm:equivalence}), there is an isotopism $(N,L,M)$, where 
\begin{align*}
& N_1=a_1x^{p^t},\quad N_4 = d_1x^{p^t},\quad N_2=N_3=0,\\ 
& L_1=a_2^{p^t}x^{p^t},\quad L_2=b_2^{p^t}x^{p^t},\quad L_3=c_2^{p^t}x^{p^t},\quad L_4=d_2^{p^t}x^{p^t},\\ 
& M_1=a_3^{p^t}x^{p^t},\quad M_2=b_3^{p^t}x^{p^t},\quad M_3=c_3^{p^t}x^{p^t},\quad M_4=d_3^{p^t}x^{p^t},
\end{align*}
where $a_1,d_1 \neq 0$.
	Then (only considering the second components), we have
\[
	L((x,y)) \star M((u,v)) = (\rule{1em}{.5pt}, \ (a_2x+b_2y)^{qQ+p^t}(c_3u+d_3v)^{p^t}+(c_2x+d_2y)^{p^t}(a_3u+b_3v)^{qQ+p^t})
\]
and 
\[
N((x,y) \ast (u,v)) = (\rule{1em}{.5pt}, \ d_1(x^{qQ}v+yu^{qQ}+(a/B)(xv^{qQ}+y^{qQ}u))^{p^t}).
\]
Comparing the coefficients of $(x^{qQ}v)^{p^t}$, $(xv^{qQ})^{p^t}$, $(x^{qQ}u)^{p^t}$ and $(xu^{qQ})^{p^t}$ yields the following four equations:
	\begin{align*}
		\left(a_2^{qQ}d_3\right)^{p^t} &= d_1 \\
		\left(c_2b_3^{qQ}\right)^{p^t} &= d_1(a/B)^{p^t} \\
		a_2^{qQ}c_3&=0 \\
		c_2a_3^{qQ} &= 0.
	\end{align*}
	The bijectivity of $L$ and $M$ induces the conditions $(a_2,c_2)\neq (0,0)$ and $(a_3,c_3) \neq (0,0)$. Thus, the last two equations only allow $a_2=a_3=0$ or $c_2=c_3=0$. Both cases contradict the first two equations. We conclude that $\P_{q,B,a}$ is not isotopic to a Zhou-Pott pre-semifield.
	
	The pre-semifields from $\Family 1$ 
	are also not isotopic to the ones from $\mathcal{CG}$, $\mathcal{G}$, 
	$\mathcal{CM}/\mathcal{DY}$, $\mathcal{ZKW}$, $\mathcal{B}_3$ by considering the order of the semifields and their nuclei (see Table~\ref{table_comm}). Furthermore, the 
	Family $\Family 1$ is not contained in $\mathcal{B}_4$ since we can choose $p,m,q$ in a way that the 
	conditions for $\mathcal{B}_4$ in Table~\ref{table_comm} are violated. 
\end{proof}

Although the parameters $p,m,q$ for the pre-semifields from Family
$\Family{1}$ are more general than that of Family $\mathcal{B}_4$,
for suitable choices of $p,m,q$ the parameters may coincide. The next proposition
shows that even in that case Family $\Family{1}$ contains new semifields thanks to
its exponential count. More precisely, we show that the number of non-isotopic
semifields from Families $\mathcal{B}_3$ and $\mathcal{B}_4$ of order $p^{3s}$ and
$p^{4s}$, respectively, is linear in $s$. 

\begin{proposition}
The number of non-isotopic pre-semifields in Family $\mathcal{B}_3$ (and $\mathcal{B}_4$ resp.) 
of order $p^{3s}$ (and $p^{4s}$ resp.) is at most $9\sigma(s)$ (and $8\sigma(s)$ resp.).
\end{proposition}
\begin{proof}

The $\mathcal{B}_4$ planar mappings are of the form
\[
	f(X) = X^{q+1} - a^{Q-1}X^{qQ+Q^3},
\]
where $a$ generates $\f{p^{4s}}^\times$. We count the number of different $a$'s which
give inequivalent planar mappings. Consider the change of variable $X \mapsto BX$,
and rescaling of $f$ to get
\[
X^{q+1} - B^{Q^3+qQ-q-1}a^{Q-1}X^{qQ+Q^3}.
\]
Note that $Q^3+qQ-q-1 = (Q-1)(Q^2+Q+q+1)$. We have by \cite[Lemma 6]{Bierbrauer10},
\[
\Gcd{Q^2+Q+q+1}{Q^3+Q^2+Q+1} = \Gcd{Q^3-q}{Q^3+Q^2+Q+1} = 4,
\]
when the semifield conditions on $q,Q$ appearing on Table~\ref{table_comm} is
satisfied.
Thus the number of inequivalent planar mappings in the Family $\mathcal{B}_4$ 
for a given $q$ is at most $4$. This means that (using Theorem~\ref{thm:eaequiv}) that for a given
$q$, the number of pre-semifields, that are not strongly isotopic, is also at most $4$. Any isotopy class
of a commutative semifield contains at most $2$ strong-isotopy classes (\cite[Theorem 2.6.]{coulterhenderson}), so
for a given $q$ there are at most $8$ non-isotopic pre-semifields.
Thus the total number of non isotopic 
pre-semifields in the Family $\mathcal{B}_4$ of order $p^{4s}$ is bounded by
$8\sigma(s)$.
The $\mathcal{B}_3$ case is essentially the same using \cite[Lemma 5]{Bierbrauer10}. In this case (again with \cite[Theorem 2.6.]{coulterhenderson})
strong isotopy and isotopy coincide.
\end{proof}

For the Family $\mathcal{ZKW}$ we are not aware of any result on the exact value
or a bound on the number of non-isotopic pre-semifields.

\begin{remark} \label{rem:q1}
We remark that we could also allow $q=1$ in $\mathcal{S}$. However, in that case
the resulting pre-semifields are strongly isotopic to Dickson semifields. 
Indeed, consider the planar mapping $F = ((1,0,0,B)_1,(0,1,A,0)_Q)$ with $A=a/B$ 
where $B$ is a non-square and $ a\in \L^\times$. Note that 
$A \notin (\M^\times)^{Q-1}$ since it is a non-square.
Define $N$ via its subfunctions $N_1=x$, $N_2=N_3=0$, $N_4 = d_1x +d_1'x^Q$ 
with $d_1=1/(1-A^{Q+1})$ and $d_1' = - A/(1-A^{Q+1})$.
Note that $x \mapsto \alpha x - \beta x^Q$ is bijective if and only if 
$\alpha/ \beta \not\in (\M^\times)^{Q-1}$. Therefore, $N_4$ is bijective
since $((A^{Q+1}-1)/(A(1-A^{Q+1}))^{Q+1}=(-1/A)^{Q+1} \neq 1$ since 
$A \notin (\M^\times)^{Q-1}$. 
We conclude that $N$ is bijective. The subfubction $N_4$ is chosen such that $d_1+A^Qd_1'=1$ and 
$Ad_1+d_1'=0$. Then
\begin{align*}
  N F &= ((1,0,0,B)_1,d_1(0,1,A,0)_Q+d_1'(0,A^Q,1,0)_Q) \\
          & = ((1,0,0,B)_1,(0,d_1+A^Qd_1',d_1A+d_1',0)_Q) = ((1,0,0,B)_1,(0,1,0,0)_Q),
\end{align*}
so $F$ is equivalent to a planar mapping belonging to a Dickson pre-semifield and the corresponding semifields are strongly isotopic by Theorem~\ref{thm:eaequiv}. It makes thus sense to exclude the case $q=1$ so that the different families do not intersect (as proven in Theorem~\ref{thm:inequiv_distinct}).
Note that the same choice of $N$ also yields equivalence between the 
Budaghyan-Helleseth planar mapping and the planar mappings associated 
with Dickson semifields for the parameter $q=Q$.
\end{remark}

\begin{remark}
Recall that Kantor \cite{Kantor03} gave a family  
that contains an exponential number of non-isotopic commutative semifields in 
characteristic two using a construction of Kantor and Williams \cite{KW}. 
We remark that Family $\Family 1$ (and in general, a planar mapping)
does not exist in characteristic two. 
However, a conceptual analogue of planar functions in characteristic two
is possible.
These are the so-called almost perfect nonlinear (APN) functions 
(whose polarizations are $2$-to-$1$) that parallel 
planar mappings 
(whose polarizations are $1$-to-$1$)
without the connection to semifields. In a follow-up work to this one, we give an
analogous method for determining equivalence of biprojective APN functions and an 
analogous family that contains an exponential number of inequivalent APN functions 
in \cite{biproj_apn}.
The first result to show that an APN family contains an exponential number of 
inequivalent functions was given recently by Kaspers and Zhou \cite{kasperszhou} 
using a different method.
\end{remark}

\section{Acknowledgments}\label{secack}

The authors would like to thank an anonymous reviewer for informing them of Theorem \ref{MP} 
which leads to a proof of the nuclei that is simpler and more coherent with the paper than 
our original proof and for several other comments that improved the presentation.

The first author was supported by the {\sf GA\v{C}R Grant 18-19087S - 301-13/201843}.
The second author was supported by the National Science Foundation under grant No. 2127742.

\subsection*{Author Contributions}\label{secaut} 
F.G.: conceptualization, methodology (lead), investigation (equal),
writing -- original draft (equal).
L.K.: methodology (supporting), investigation (equal),  
writing -- original draft (equal).

\bibliographystyle{amsplain}
\bibliography{semifields}

\bigskip
\hrule
\bigskip
\end{document}